\documentclass[10pt]{amsart}
\usepackage{amssymb}
\usepackage{mathrsfs}%

\newtheorem{thm}{Theorem}[section]%
\newtheorem{prop}[thm]{Proposition}%
\newtheorem{lem}[thm]{Lemma}
\newtheorem{cor}[thm]{Corollary}
\newtheorem{thmintro}{Theorem}

\newtheorem{propintro}[thmintro]{Proposition}

\theoremstyle{definition}%
\newtheorem{example}[thm]{Example}%
\newtheorem{rem}[thm]{Remark}%
\newtheorem{question}[thmintro]{Question}%

\newcommand{\ats}{\mathscr{U}}
\newcommand{\catname}[1]{\mbox{\normalfont{#1}}}

\calclayout

\usepackage{hyperref}
\usepackage{cleveref}
\usepackage{tikz-cd}
\usepackage{enumerate}
\usepackage{mathtools}

\crefname{thm}{theorem}{theorems}
\Crefname{thm}{Theorem}{Theorems}
\crefname{prop}{proposition}{propositions}
\Crefname{prop}{Proposition}{Propositions}
\crefname{lem}{lemma}{lemmas}
\Crefname{lem}{Lemma}{Lemmas}

\DeclareMathOperator{\op}{op}
\DeclareMathOperator{\id}{Id}
\DeclareMathOperator{\ima}{im}

\DeclareMathOperator{\dom}{dom}

\DeclareMathOperator{\Ker}{Ker}

\DeclareMathOperator{\pt}{pt}

\begin{document}

\title{A generalisation of the Munn semigroup}

\author{Francesco Tesolin}

\address{F. Tesolin, Department of Mathematics and the Maxwell Institute for Mathematical Sciences, 
          Heriot-Watt University,
          Riccarton, Edinburgh, EH14 4AS, UK}

	  \email{ft2021@hw.ac.uk}

	\maketitle

\begin{abstract}
	To each meet-semilattice $E$ is associated an inverse semigroup $T_{E}$ called the Munn semigroup of $E$. We
generalise this construction by replacing the meet-semilattice $E$ by a presheaf of sets $X$ over a meet-semilattice.
The inverse semigroup
$T_{X} $ that results is called the generalised Munn semigroup. Our construction can be viewed as a generalisation of one
due to Zhitomirskiy as well as a restriction of one due to Reilly. We prove that idempotent-separating representations in to the
generalised Munn semigroup characterise étale actions of inverse semigroups.
\end{abstract}

\small



\section{Introduction}

Actions of inverse semigroups can be seen as the usual semigroup action, however étale actions, or 
\emph{supported actions} as we call them in this paper,
are more significant in inverse semigroup theory. As noted in \cite{LawKud2014}, since the idempotents
of an inverse semigroup can be regarded as partial identities, we want to consider an action on a collection of sets
where each idempotent fixes one of these sets. A supported action is then viewed as an action on a presheaf over the idempotents.
Supported actions were first introduced by Steinberg \cite{steinberg2009strong} to study a Morita theory for inverse semigroups.
This work formed part of a wider program of studying C*-algebras arising from inverse semigroups, a subject which has grown
significantly in recent years.

To study supported actions we define a notion of partial transformations on a presheaf.
We are not the first: Renault defined an inverse semigroup of partial isomorphisms of a
sheaf \cite{RenaultBook}, Reilly defined an inverse semigroup of order-isomorphisms between principal order-ideals
of a pre-semilattice \cite{reilly1977enlarging}, and Zhitomirskiy defined an inverse semigroup of
principal partial automorphisms of a (topological) bundle \cite{zhito1}.
The two main questions we wish to answer in this paper are:
\begin{question}
How do these three structures relate to one another?
\end{question}

\begin{question}
For an inverse semigroup $S$, 
how do the supported actions of $S$ relate to representations of $S$ by partial transformations on a presheaf?
\end{question}

Our generalised Munn semigroup can be seen as a generalisation of Renault's inverse semigroup. We generalise the Munn
representation theorem \cite[Lemma 3.1]{MunnSemigroup}
to answer Question~B.

\begin{thmintro}[Generalised Munn representation]
	Let $(X,S,p)$ be a supported action and let $Y=(X,E(S),p)$ be the presheaf obtained by restricting the action $(X,S,p)$ 
	to the idempotents.
	Then, there is an idempotent-separating 
	representation $\xi\colon S \to T_{Y}$
	 whose image is a wide inverse subsemigroup of $T_{Y}$.
\end{thmintro}

The previous theorem shows that our construction is sufficient for the work
of Reilly, recapturing their embedding theorem \cite[Theorem 3.1]{reilly1977enlarging}.

Theorem C proved that for every supported action we have a representation, the next theorem proves the converse: that 
idempotent-separating homomorphisms into $T_{X}$ give rise to supported actions.

\begin{thmintro} 
	Let $S$ be an inverse semigroup and $(X,E(S),p)$ a presheaf. Then, the following are equivalent:
	\begin{enumerate}
	\item There exists an idempotent-separating homomorphism $\phi\colon S \to T_{X}$ whose image is a 
	wide inverse subsemigroup of $T_{X}$.
	\item 
	There exists a supported action $(X,S,\overline{p})$ such that the presheaf obtained by restricting 
	to the idempotents is isomorphic to $(X,E(S),p)$.
	\end{enumerate}
\end{thmintro}

This answers Question $B$. With the following results we answer Question $A$.
We link our work with that of Zhitomirskiy \cite{zhito1}, who for a surjective continuous map $\pi\colon X \to B$
defined the inverse semigroup $\mathbf{La}(\pi)$ of partial homeomorphisms on $X$, respecting the structure of $\pi$.
This can be seen
as a topological version of our work.
Using the Munn representation theorem, we prove the following proposition showing how
for a (topological) space $X$ 
the Munn semigroup on the open sets $\Omega(X)$
is connected to the inverse semigroup of partial homeomorphisms on $X$, denoted $\mathcal{I}(X, \tau)$.

\begin{propintro}
	Let $X$ be a sober space. Then, $T_{\Omega(X)} \cong \mathcal{I}(X, \tau)$.
\end{propintro}

Using Proposition E and the generalised Munn representation we define an isomorphism between
$\mathbf{La}(\pi)$ and $T_{\Gamma(\pi)}$, where $\Gamma(\pi)$ is the sheaf of sections  of $\pi$, and
$\pi $ is an étale bundle over a sober space.

\begin{thmintro}
	Let $\pi\colon X \to B$ be an étale bundle such that $B$ is a sober space. Then,
	\[
		T_{\Gamma(\pi)} \cong \mathbf{La}(\pi).
	\]
\end{thmintro}

\noindent
\textbf{Organisation of the paper.}
In Section 3 we define supported actions, prove some of their first properties and give many examples including:
as group actions, characterising ideals and normal
subsemigroups, in semigroup cohomology, and in actions of Clifford semigroups.
In Section 4 we give an explicit proof that presheaves can be viewed as supported actions. We define subpresheaves and the 
inverse semigroup of isomorphisms between subpresheaves.
In Section 5 we define the generalised Munn semigroup, show it really is a generalisation, prove Theorem C and Theorem D, and
look at some requirements for the generalised Munn representation to be faithful.
In Section 6 we introduce frames and spaces and use properties of sober spaces \cite[Chap. 9]{sheaves} 
to prove Proposition $E$. We then use this result to prove Theorem F.
In Section 7 we show how bundles can also arise naturally from presheaves.

\section{Preliminaries}
See \cite{LawsonBook} for background on inverse semigroups.
In an inverse semigroup $S$ define, in the usual way, $\mathbf{d}(s) = s^{-1}s$ and $\mathbf{r}(s) = ss^{-1}$.
Denote the set of idempotents of $S$ by $E(S)$.
We call an abelian inverse semigroup with all elements being idempotent a
\emph{meet-semilattice}. 
We denote the set of elements of $S$ which commute with all elements of $E(S)$ by $Z(E(S))$.
An inverse semigroup $S$ such that $S = Z(E(S))$ is called \emph{Clifford}.
All inverse semigroups have a natural partial order which will be used throughout.
An inverse subsemigroup of $S$ containing all idempotents of $S$ is said to be \emph{wide}.
A subsemigroup $M$ is called \emph{self-conjugate} if $s^{-1}Ms \subseteq M$ for all $s \in S$.
Wide self-conjugate inverse subsemigroups are called \emph{normal}.
Define the \emph{Kernel} of the congruence $\rho$ of an inverse semigroup, denoted $\Ker \rho$,
to be the union of the $\rho$-classes containing idempotents. As usual, the congruence induced on $S$ by
the homomorphism $\theta\colon S \to T$ is called the kernel 
of $\theta$ and is denoted by $\ker \theta$.
We say that a congruence $\rho$ on $S$ is \emph{idempotent-separating} when each congruence class contains at most
one idempotent.
The inverse semigroup of all partial bijections on a set $X$ is called the \emph{symmetric inverse monoid} and is denoted by
$\mathcal{I}(X)$. A homomorphism from an inverse semigroup $S$ to a symmetric inverse monoid is called a \emph{representation}
of $S$, if the homomorphism is injective the representation is said to be faithful.
Define the relation $\mu$, on an inverse semigroup $S$, by
    \[
(a,b) \in \mu \iff  \ a^{-1}ea = b^{-1}eb \ \forall e \in E(S).
    \]

\begin{prop}[{\cite[Ch. 5.2 Prop. 3-4]{LawsonBook}}]\label[proposition]{MaxCongruence}
    Let $S$ be an inverse semigroup. Then:
    \begin{enumerate}[1.]
    	\item  
    $\mu$ is the maximum idempotent-separating congruence on $S$.
    \item $\mu$ is the unique idempotent-separating congruence such that $\Ker \mu = Z(E(S))$.
    \end{enumerate}
     \end{prop}

An inverse semigroup is said to be \emph{fundamental} if $\mu$ is the equality relation.

\begin{prop}[{\cite[Ch. 5.2 Prop. 5]{LawsonBook}}]\label[proposition]{fundamentalprop}
	Let $S$ be an inverse semigroup. Then:
    \begin{enumerate}[1.]
        \item $S$ is fundamental if and only if $Z(E(S)) = E(S)$.
        \item $S/ \mu$ is fundamental.
    \end{enumerate}
\end{prop}

First introduced in \cite{MunnSemigroup}, the \emph{Munn semigroup} of a meet-semilattice $E$ is the inverse semigroup
$T_{E}$ consisting of all order-isomorphisms between principal order-ideals of $E$. Note that the idempotents of
$T_{E}$ are isomorphic to $E$.
There is the following representation
theorem from an inverse semigroup to the Munn semigroup of its meet-semilattice of idempotents.

\begin{thm}[{Munn representation, \cite[Lemma 3.1]{MunnSemigroup}}] \label{classical munn}
	Let $S$ be an inverse semigroup. Then there is an idempotent-separating homomorphism $\delta\colon S \to T_{E(S)}$ such that the $\ker\delta = \mu$ whose image is a wide inverse subsemigroup of $T_{E(S)}$.
\end{thm}

Observe that for a meet-semilattice $E$ the greatest lower bound of any pair of elements $e,f \in E$ in the natural partial order is $ef$. This is known as the
\emph{meet} of $e$ and $f$. Since all semilattices in this paper are meet-semilattices we will refer to them simply as semilattices.

Fix the following notation for functions: For functions $f\colon X \to Y$ and $g\colon Y \to Z$ the composition
$x \mapsto g(f(x))$ is denoted by the concatenation $gf$.
For any function $f\colon X \to Y$ the restriction of $f$ to a subset $X' \subseteq X$ is denoted by $f|_{X'}$.
The identity function on a set $X$ is denoted by $\id_{X}$. The domain of a function $f\colon X \to Y$ is denoted $\dom f = X$ and the image is denoted
$\ima f = f(X)$.

A \emph{presheaf of sets over a semilattice $E$} is  a collection of
sets $X_{e}$ for each $e \in E$ and maps $\phi^{e}_{f}\colon X_{e} \to  X_{f}$, for all $e \geq f$, called 
\emph{restriction maps}, such that
\[
	\phi^{e}_{e} = \id_{X_{e}}  \mbox{ and } \ \phi^{f}_{g} \phi^{e}_{f} = \phi^{e}_{g} \ \mbox{ for all } e \geq f \geq g 
.\]

Denote a presheaf over a semilattice $E$ by $(X_{e}, \phi^{e}_{f})$. A 
homomorphism from $(X_{e}, \phi^{e}_{f})$ to $(Y_{e}, \psi^{e}_{f})$ is a
\emph{natural transformation}: a collection of
functions $(\eta_{e}\colon X_{e } \to Y_{e})_{e \in E}$ such that the following square commutes for all $f \leq e$.
\begin{equation}\label{natural}
\begin{tikzcd}
		X_{e} \ar[r, "\eta_{e}"] \ar[d, "\phi^{e}_{f}"'] & Y_{e} \ar[d, "\psi^{e}_{f}"] \\
		X_{f} \ar[r, "\eta_{f}"] & Y_{f}
	\end{tikzcd}
\end{equation}

\section{Étale actions}

We recall the notion of an \emph{étale right action} introduced in \cite{funk2010universal}. We prefer to use
the term \emph{supported right action}, as in \cite{lawson2021morita}, to avoid overloading the term `étale'.
Let $X$ be a set and $S$ an inverse semigroup. A \emph{supported right action} of $S$ on $X$ is
a pair of functions
	 \[ X \times S \to  X, \ \mbox{denoted by} \ (x,s) \mapsto x \cdot s, \
 \mbox{and} \ p\colon X \to E(S) \]
satisfying the following three conditions for all $x \in X$ and $s,t \in S$:
\begin{enumerate}[({SA}\arabic*)]
	\item $(x \cdot s) \cdot t = x \cdot st.$ 
	\item $x \cdot p(x) = x$.
	\item $p(x \cdot s) = s^{-1}p(x)s$.
\end{enumerate}
We denote a supported right action by a triple $(X,S,p)$.
If $p$ is surjective, we say that $(X,S,p)$ has \emph{global support} or is a \emph{globally} supported right action.
In this paper we only deal with right actions and therefore will refer to them as supported actions.
Given two supported actions $(X,S,p)$ and $(Y,T,q)$ a \emph{homomorphism} $(X,S,p) \to (Y,T,q) $ is a pair $(\alpha, \theta)$ 
consisting of
a function $\alpha\colon X \to Y$ and a homomorphism $\theta\colon S \to T$ where $\alpha(x \cdot s) = \alpha(x) \cdot \theta(s)$, for all $x \in X$ and $s \in S$,
and such that
the following diagram commutes.
\[
	\begin{tikzcd}
		X \ar[d, "p"'] \ar[r, "\alpha"]  & \ar[d, "q"] Y \\
						 E(S) \ar[r, "\theta|_{E(S)}"']& E(T)
	\end{tikzcd}
\]

By fixing an inverse semigroup $S$ one may consider the category of all supported actions of $S$, 
called the \emph{classifying topos} of $S$, 
introduced in \cite{funk2010universal}. The morphisms of this category are homomorphisms $(\alpha, \theta)$ such that
$\theta $ is the identity on $S$. In \cite{LawStein2004}, it was
proved explicitly that the classifying topos is equivalent to a presheaf category. We recall this result now to give justification
for studying supported actions. See \cite{sheaves} for background on category theory.
Let $S$ be an inverse semigroup. Define $C^{l}(S)$, the \emph{(left) Leech category} of $S$, to be the category objects being the idempotents of $S$
and morphisms $e \to f$ the pairs $(f,a)$ with $f \in E(S)$ and $a \in S$  satisfying $e = \mathbf{d}(a)$ and
$\mathbf{r}(a) \leq f$. Composition is defined by $(g,b)(f,a) = (g,ba)$.
This category was first defined in \cite{Loganathan}, and further studied by Leech in \cite{LeechConstruction}.
Recall that for an arbitrary small category $C$, an \emph{action of $C $}, that is a functor from $C^{\op}$ to the category
of sets, is called a \emph{presheaf of sets over $C$}.

\begin{thm}[{\cite[Theorem 4.3]{LawStein2004}}] \label{Leechcat}
	Let $S$ be an inverse semigroup.
	The category of supported actions of $S$ is equivalent to the category of actions of the Leech category $C^{l}(S)$.
\end{thm}

Steinberg \cite{steinberg2009strong} defined a partial order on each supported action as follows:
$x \leq y$ if and only if $x= y \cdot p(x)$.
\begin{lem}[{\cite[Proposition 3.2]{steinberg2009strong}}] \label[lemma]{steinlem}
	Let
    $(X,S,p)$ be a supported action. Then, the following are equivalent:
    \begin{enumerate}[1.]
        \item $x = y \cdot p(x)$.
        \item $x = y \cdot e$  for some $e\in E(S)$.
    \end{enumerate}
    Moreover, the partial order $\leq$ is compatible with the action of $S$, and $p$ is order preserving.
\end{lem}

Let $(X,S,p)$ be a supported action. Define a \emph{subaction} of $(X,S,p)$ to be a subset $Y$ of $X$ that is closed under
the action $X \times S \to X$. Then,
$(Y,S,p|_{Y})$ is a supported action.

\begin{lem} \label[lemma]{subact ideal}
	Let $(X,S,p)$ be a supported action. Every subaction of $(X,S,p)$ is an order-ideal of $X$.
\end{lem}

\begin{proof}
	Let $Y$ be a subaction of $(X,S,p)$. Let $x \in X$ and $y \in Y$ such that $x \leq y$. Then, 
	since $Y$ is closed under the action of $S$ and by definition of the partial order,
	$x =y \cdot p(x) \in Y$. 
\end{proof}

We now give some examples of where supported actions arise.

\begin{example} 
	Let $G$ be a group. Then $G$ has a unique idempotent and hence for any supported action the map $p\colon X \to E(G)$ is
	trivial. Therefore, every group action is a globally supported action.
\end{example}

\begin{example} \label[example]{clifford subaction}
	Let $S$ be an inverse semigroup. Define an action $E(S) \times S \to E(S)$ by $e \cdot s = s^{-1}es$.
	Then, $(E(S),S,\id_{E(S)})$ is a supported action. We now show the result,
	as stated in \cite{steinberg2009strong}, that this supported action
	is the terminal object in the classifying topos of $S$.
	Let $(X,S,p)$ be a supported action. Then, the pair $(p, \id_{S})$ is a homomorphism from
	$(X,S,p)$ to $(E(S),S,\id_{E(S)})$ since $p(x \cdot s) = s^{-1}p(x)s = p(x) \cdot s$ and
	$\id_{E(S)} p = p$. Suppose $(f, \id_{S})$ was another such homomorphism, then
	$f = \id_{E(S)} f = p$.
	Therefore, $p\colon X \to E(S)$ is the unique such homomorphism.
\end{example}

\begin{example} \label[example]{right ideal}
	Let $S$ be an inverse semigroup. Define an action $S \times S \to S $ by $s \cdot t = st$ and map
$p\colon S \to E(S)$ by
	$p(s) = \mathbf{d}(s)$. Note that $s \cdot p(s) = s \cdot \mathbf{d}(s) = s$ and 
	\[
		p(s \cdot t) = \mathbf{d}(st) = t^{-1}s^{-1}st = t^{-1} \mathbf{d}(s) t = t^{-1}p(s) t
	.\]
Therefore, $(S,S,p)$ is a supported action. Let $I$ be a subaction of $(S,S,p)$. Then, by definition, $I$ is closed
under composition on the right by $S$. Hence, $I$ is a right ideal of $S$. Observe that the right ideals of $S$ are 
therefore exactly the
subactions of $(S,S,p)$.
\end{example}

\begin{example} 
Let $S$ be an inverse semigroup. We define a \emph{supported action on a Clifford semigroup} to be a supported 
action $(A,S,p)$ where $A$ is a Clifford semigroup, $p$ is an idempotent-separating surjective homomorphism, and
$S$ acts by endomorphisms on $A$, that is
\[
	(a_1  a_2)\cdot s = (a_1 \cdot s) (a_2 \cdot s)
\]
for all $a_1,a_2 \in A$ and $s \in S$.
Note that, as in \Cref{Leechcat}, a supported action of $S$ on a Clifford semigroup is a equivalent to a presheaf of groups over $C^{l}(S)$, the
Leech category of $S$.
Let $S$ be an inverse semigroup. It is well-known, \cite[Chap. 5.2]{LawsonBook},
that $Z(E(S))$ is a Clifford semigroup and a normal subsemigroup of $S$.
Define
	$p\colon Z(E(S)) \twoheadrightarrow E(S)$ by $p(n) = \mathbf{d}(n)$. Then, $p$ restricted to the idempotents
	is bijective. Let $n_1,n_2 \in Z(E(S))$. Then, since $n_2$ commutes with the idempotents,
	$\mathbf{d}(n_1n_2) = n_2^{-1}\mathbf{d}(n_1)n_2
	= \mathbf{d}(n_1) \mathbf{d}(n_2)$. Hence, $p$ is an idempotent-separating surjective homomorphism. Define an action
	$Z(E(S)) \times S \to Z(E(S))$ by $n \cdot s = s^{-1} n s$.
	Since $Z(E(S))$ is closed under conjugation the action is a well-defined function. Since each
	$n \in Z(E(S))$ commutes with all $e \in E(S)$. Then
	\[
		n \cdot p(n) = n \cdot \mathbf{d}(n) = \mathbf{d}(n)^{-1} n   \mathbf{d}(n) = n\mathbf{d}(n)^{-1}   \mathbf{d}(n) = n
		\mathbf{d}(n) = n
	\]
	and 
	\[
		p(n \cdot s) = \mathbf{d}(s^{-1} n s) = s^{-1}n^{-1}(ss^{-1})ns= s^{-1} n^{-1}n (ss^{-1}) s = s^{-1} \mathbf{d}(n) s =
		s^{-1} p(n) s
	\]using the fact that $\mathbf{d}(n)$ and $ss^{-1}$ are in $E(S)$. Let $n_1,n_2 \in N$ and $s \in S$. Then, since
	$n_2$ commutes with all the idempotents,
	 \[
		 (n_1 \cdot s)(n_2 \cdot s) = s^{-1} n_1 ss^{-1} n_2 s = s^{-1} n_1 n_2 (ss^{-1}) s = (n_1n_2)\cdot s
	 .\]
	Therefore, $(Z(E(S)), S,p)$ is a supported action on a Clifford semigroup.

	A subaction of a supported action on a Clifford semigroup are those subactions $N \subseteq S$ which
	are Clifford semigroups and for which $p$ restricted to $N$ is still idempotent-separating and surjective.
	Let $N$ be a subaction of $(Z(E(S)), S, p)$. Then, $N$ is a normal inverse subsemigroup since
	it is closed under the action of $S$ and $p|_{N}$ is surjective.
	By \cite[Chap. 5.2]{LawsonBook}
	$N$ is a subaction of  $Z(E(S))$ if and only if $N$ is the Kernel of an idempotent-separating congruence on $S$.
\end{example}

\begin{example}
Let  $S$ be an inverse semigroup.
Let $(A,S,p)$ be a supported action on a Clifford semigroup $A$ where $A$ is abelian.
 We call such an action
a supported action on an abelian inverse semigroup. Specialising \Cref{Leechcat} this corresponds to a presheaf
of abelian groups over $C^{l}(S)$.
In \cite{Lausch1975} the notion of $S$-module is used to define a cohomology
theory of inverse semigroups. Let $S$ be an inverse semigroup. A \emph{right $S$-module} is an abelian inverse semigroup $A$  with an action
    $A \times S \to A$  and an isomorphism $\theta\colon E(S) \to E(A)$ denoted by $e \to 0_e$, satisfying:
    \begin{enumerate}[1.]
        \item $(a_1 + a_2)\cdot s = a_1 \cdot s +a_2 \cdot s$ for $a_1,a_2 \in A$ and $s \in S$.
        \item $(a \cdot  s_1) \cdot s_2 = a \cdot (s_1 s_2)$ for $a \in A$ and $s_1,s_2 \in S$.
        \item $a \cdot e = a + 0_e$ for $a \in A$ and $e \in E(S)$.
        \item $0_e \cdot s = 0_{s^{-1}es}$ for $e \in E(S)$ and $s \in S$.
    \end{enumerate}
We show that $S$-modules are the same as supported actions on an abelian inverse semigroup as first
noted in \cite{steinberg2023twists}.
Let $(A,S,p)$ be a supported action. 
Since  $p$ restricts to an isomorphism on the idempotents denote
by  $0_{e}$ the unique idempotent of $A$ such that $p(0_{e}) = e$ for each $e \in E(S)$. All conditions to be an $S$-module
are then already assumed
except (3) and (4).
By (1) the action of $S$ maps idempotents of $A$ to idempotents.
By (SA3)
$p(0_{e} \cdot s) = s^{-1}es$ hence, since $p$ is idempotent-separating, $0_{s^{-1}es} = 0_{e} \cdot s$ for all
$e \in E(S)$ and $s \in S$. So (4) holds. In particular, $0_{e} \cdot f = 0_{f} \cdot e$ for all $e,f \in E(S)$.
    Let $a \in A$. By (1), $\mathbf{d}(a) = \mathbf{d}(a) \cdot p(a)$. Hence,
    \[
	    p(\mathbf{d}(a)) \leq p(a) = p(0_{p(a)}).
    \]
   Since $p|_{E(A)}$ is an isomorphism then $\mathbf{d}(a) \leq 0_{p(a)}$ and 
    $a = a + 0_{p(a)}$. 
   Let $e \in E(S)$. Using that $p(a + 0_{e}) = p(a)e$ it is shown that    \[
	    a \cdot e = ( a + 0_{p(a)}) \cdot e = (a + 0_{e}) \cdot ep(a) = (a + 0_{e}) \cdot p(a + 0_{e}) = a + 0_{e}.
    \]
\end{example}

\begin{example}
	In \cite{steinberg2023twists,Lausch1975} it is shown that idempotent-separating extensions give rise to supported right
	actions on abelian inverse semigroups. We recall the construction now.
Let $j\colon T \to S$ be an idempotent-separating surjective homomorphism such that $\Ker j$ is an abelian inverse semigroup.
Let $k\colon S \to T$ be a set-theoretic section of $j$.
Define a function $\Ker j \times S \to \Ker j$ by
$t \cdot s = k(s)^{-1}tk(s)$. Note that $j(k(s)^{-1}tk(s)) = s^{-1}j(t)s \in E(S)$ for all $t \in \Ker j$ and $s \in S$ so
the function is well-defined.
We work to show that $(\Ker j, S, j)$ with this function is a supported action on an abelian inverse semigroup.
Note that (SA1) and (SA3) follow immediately from the 
definition of the action, and by assumption $j$ is a idempotent-separating surjective homomorphism. 
Let $t \in \Ker j$ and denote  $s = k(j(t))$. Since $j(s) = j(t)$ it follows that $s \in \Ker j$ and
$\mathbf{d}(s) = \mathbf{d}(t)$. Hence, since $\Ker j$ is abelian,
\[
	t \cdot j(t) = s^{-1}ts = t s^{-1}s = t \mathbf{d}(t) = t
.\]
So (SA2) holds. 
The last condition to check is that $S$ acts by endomorphisms on $\Ker j$.
Let $t_1,t_2 \in \Ker j$ and $s \in S$. Then, by \Cref{MaxCongruence} $\Ker j \subseteq Z(E(T))$, hence
\[
	(t_1t_2)\cdot s = k(s)^{-1}(t_1t_2)k(s) = (k(s)^{-1}t_1k(s))(k(s)^{-1}t_2k(s)) = (t_1 \cdot s)(t_2 \cdot s)
.\]
We now show that this action does not depend on the section chosen. Let $k_1,k_2: S \to T$ be two set-theoretic sections
of $j$ and $s \in S$.  Then, $j(k_1(s)) = j(k_2(s)) = s$ and
$k_1(s)k_2(s)^{-1} \in \Ker j$, denote $h = k_1(s)k_2(s)^{-1}$. 
Since $(k_1(s),k_2(s)) \in \ker j $ it implies that $\mathbf{d}(k_1(s)) = \mathbf{d}(k_2(s))$.
Therefore, $k_1(s) = k_1(s)k_2(s)^{-1}k_2(s) = hk_2(s)$.
Let $t \in \Ker j $. Then, since $\Ker j$ is abelian and $\mathbf{d}(h) = \mathbf{r}(k_2(s))$,
\[
	k_1(s)^{-1}tk_1(s) = k_2(s)^{-1}h^{-1}thk_2(s) = k_2(s)^{-1}th^{-1}hk_2(s) = k_2(s)^{-1}tk_2(s)
.\]
\end{example}

\begin{example}
	This example was described in \cite{LawKud2014}. Note that every supported action $(X,S,p)$ can be
	restricted to the action of the idempotents to get a supported action $(X, E(S),p)$. This is a forgetful functor
	from supported actions of $S$ to supported actions of $E(S)$. In \cite{LawKud2014} it was shown
	there exists a left adjoint to this 
	forgetful functor; we call this the \emph{free supported action} which we now 
	describe.
	Let $S$ be an inverse semigroup. Let $(X,E(S),p)$ be a supported action of $E(S)$ on $X$ with action map
	$X \times E(S) \to X$. Define 
	\[
		S \ast X = \{(s,x) \in S \times X : \mathbf{r}(s) = p(x)\}
	\] and $d\colon S \ast X \to E(S)$ by $d(s,a) = \mathbf{d}(s)$. Define $(S \ast X) \times S \to S \ast X $ by
	$(t,x) \cdot s = (ts,x \cdot \mathbf{r}(ts))$ which is well-defined because $\mathbf{r}(ts) \leq \mathbf{r}(t)$.
	We now that show that this does in fact define a supported action.
	(SA1) follows from $\mathbf{r}(ts_1s_2) \leq \mathbf{r}(ts_1)$ for all $t,s_1,s_2 \in S$. Let $(t,x) \in S \ast X$. Then,
	(SA2) holds since
	\[
		(t,x) \cdot \mathbf{d}(t) = (t, x \cdot \mathbf{r}(t)) = (t,x \cdot p(x)) =(t,x).
	\]
	Axiom (SA3) holds since, for
	all $t,s \in S$ and $x \in X$,
	\[
		d(ts, x \cdot \mathbf{r}(ts)) = \mathbf{d}(ts) = s^{-1} \mathbf{d}(t)s = s^{-1} d(t, x)s
	.\]
	Therefore, $(S \ast X, S, d)$ is a supported action.

\begin{example}
	Given a presheaf of groups over a semilattice it is well-known \cite{Clifford1941} that one
can construct a Clifford semigroup called the \emph{strong semilattice of groups} and that every 
Clifford semigroup is isomorphic to a strong semilattice of groups,
see \cite[Chap. 5.2]{LawsonBook}. 
We extend this result to supported actions of a Clifford semigroup.
Define a \emph{group action} to be a triple $(X,G, \circ)$ such that $X$ is a set, $G$ is a group, and 
$\circ\colon X \times G \to X$ defines a group action of $G$ on $X$. A \emph{homomorphism of group actions} from
$(X, G, \circ_{G} )$ to $(Y, H, \circ_{H})$ is a pair
$(\phi, \psi) $ consisting of a function $\phi\colon X \to Y$, and a group homomorphism $\psi\colon G \to H$ such that
$\phi(x \circ_{G} g) = \phi(x) \circ_{H} \psi(g)$ for all $x \in X $ and $g \in G$.
A  \emph{presheaf of group actions over a semilattice} $E$ is a collection of group actions
$(X_{e}, G_{e}, \circ_{e})$ for each $e \in E$ and a collection of homomorphisms of group actions
$(\phi^{e}_{f}, \psi^{e}_{f})$ for all $f \leq e$
such that: $\phi^{e}_{e} = \id_{X_{e}}$, $\psi_{e}^{e} = \id_{G_{e}}$, and for all $h \leq f \leq e$ 
\[
	\phi^{f}_{h} \phi^{e}_{f} = \phi^{e}_{h} \quad \mbox{ and } \quad \psi^{f}_{h} \psi^{e}_{f} = \psi^{e}_{h}
.\]
Let $X = \bigsqcup_{e \in E} X_{e}$ and $S = \bigsqcup_{e \in E} G_{e}$. Then, $S$ is a Clifford semigroup with product defined by
$s \otimes t = \phi^{e}_{ef}(s)\phi^{f}_{ef}(t)$, where $s \in G_{e}$ and $t \in G_{f}$.
Define, $p\colon X \to E(S)$ by $p(x) = 1_{e}$ the identity of $G_{e}$, where $x \in X_{e}$, and $\cdot\colon X \times S \to X$ by
$x \cdot s = \phi^{e}_{ef}(x) \circ_{ef} \psi^{e}_{ef}(s)$, where $x \in X_{e}$ and $s \in G_{f}$.
Let $x \in X_{e}$, $s \in G_{f}$, and $t \in G_{h}$. Then, by the definition of homomorphism of group actions, (SA1) follows from:
\[
	\phi^{ef}_{efh}(\phi^{e}_{ef}(x) \circ _{ef} \psi^{f}_{ef}(s)) \circ_{efh} \psi^{h}_{efh}(t)
	= (\phi^{e}_{efh}(x) \circ_{efh} \psi^{f}_{efh}(s)) \circ_{efh} \psi^{h}_{efh}(t) = x \cdot (s \otimes t)
.\]
Moreover, (SA2) and (SA3) follow from
 \[
	x \cdot  p(x) = \phi^{e}_{e}(x) \circ_{e} 1_{e} = x \circ_{e} 1_{e} = x
\]
and
\[
	p(x \cdot s) = p( \phi^{e}_{ef}(x) \circ_{ef} \psi^{f}_{ef}(s)) = 1_{ef} = s^{-1} \otimes 1_{e} \otimes s
.\]
Therefore, from a presheaf of group actions over a semilattice we constructed a supported action $(X,S,p)$,
we call this action 
a \emph{strong semilattice of group actions}.
We show that every supported action of a Clifford semigroup is isomorphic to a strong semilattice of group actions.
	Let $(X,S,p)$ be a supported action such that $S$ is Clifford. Then,
	$S = \bigcup_{e \in E(S)} G_{e} $ is a disjoint union of groups $G_{e} = \{s \in S : \mathbf{d}(s) = e\}$.
	Let $x \in X_{e}$ and $s \in G_{e}$. Then,
	\[
		p(x \cdot s) = s^{-1}p(x)s = p(x)s^{-1}s = ee = e.
	\]
	Hence, for each $e \in E(S)$ we can restrict the action $X \times S \to X$ to the 
	action $\circ_{e}\colon X_{e} \times  G_{e} \to X_{e}$ and this is a group action.
	For all $f \leq e$
	define $\phi^{e}_{f}\colon X_{e} \to X_{f}$ by $x \mapsto x \cdot f$ and $\psi^{e}_{f}\colon G_{e} \to G_{f}$ by
	$s \mapsto sf$. Then, the collection of $(X_{e}, G_{e}, \circ_{e})$ and $(\phi^{e}_{f}, \psi^{e}_{f})$ for all
	$f \leq e$ defines a presheaf of group actions over the semilattice $E(S)$. 
	The associated strong semilattice of group actions is isomorphic to $(X,S,p)$.
\end{example}
\end{example}

\section{Presheaves}
By fixing a semilattice $E$, one may consider the category of presheaves over $E$ with morphisms being natural
transformations.
In \cite[p. 393]{Goldblatt} it was first noted that a presheaf over $E$ can be seen as a supported action of $E$.
This is a special case of \Cref{Leechcat} since the category $C^{l}(E)$ is the usual poset category of a semilattice. We give
an explicit proof below.

\begin{thm}\label{NoFunc}
	Let $E$ be a semilattice. 
	The category of presheaves over $E$ is equivalent to the category of supported actions of $E$.
\end{thm}

\begin{proof}
Let $(X_{e}, \phi^{e}_{f})$ be a presheaf over $E$.
	Define $X \coloneqq \bigsqcup_{e \in E} X_{e}$ the disjoint union of all the $X_{e}$, and $p\colon X \to E$ by
	$p(x) = e$ if $x \in X_{e}$.
	Define the action $X \times E \to X$ by 
	$x \cdot f = \phi^{e}_{fe}(x)$ where $x \in X_{e}$. Note that
	\[
		(x \cdot f) \cdot h = \phi^{fe}_{hfe} \phi^{e}_{fe}(x) = \phi^{e}_{hfe}(x) = x \cdot fh.
	\] Moreover, $x \cdot p(x) = \phi^{e}_{e}(x) = \id_{X_{e}}(x) = x$ and
	$p(x \cdot f) = ef = p(x)f$. Hence, $(X,E,p)$ is a supported action.
	Conversely, suppose $(X,E,p)$ is a supported action and define $X_{e} \coloneqq p^{-1}(e)$ and $\phi^{e}_{h}
	(x) = x \cdot h$ for all $e \in E$ and $h \leq e$. Then, if $g \leq f \leq e$ by definition $fg = g$ and hence
	\[
		\phi^{f}_{g} \phi^{e}_{f}(x) = (x \cdot f) \cdot g = x \cdot fg = x \cdot g =  \phi^{e}_{g}(x)
	.\]
	Note that $\phi^{e}_{e}(x) = x \cdot p(x) = x$ for all $x \in X_{e}$, so $\phi^{e}_{e}=\id_{X_{e}}$.
	Observe  that these constructions are inverse to each other.
	Let $(\eta_{e})_{e \in E}$
	be a natural transformation from $(X_{e}, \phi^{e}_{f})$ to $(Y_{e}, \psi^{e}_{f})$.
	Then,
	for $X = \bigsqcup_{e \in E} X_{e}$ and $Y = \bigsqcup_{e \in E} Y_{e}$ define a map
	$\eta\colon X \to Y $ by $\eta(x) = \eta_{e}(x)$ for $x \in X_{e}$. By definition
	$ \eta(x) \in Y_{e}$ and hence $p \eta (x) = p(x)$.
	By the diagram (\ref{natural}) if $x \in X_{e}$ and $f \in E$ then
	\[
		\eta(x \cdot f) = \eta(\phi^{e}_{ef}(x)) = \psi^{e}_{ef}(\eta(x)) = \eta(x) \cdot f
	.\]
	Conversely, for a homomorphism of supported actions $\alpha\colon X \to Y$ define
	$\alpha_{e}$ to be the restriction of $\alpha $ to $p^{-1}(e) = X_{e}$. As above, it is simple to
	check that this is a natural transformation.
\end{proof}

Let $(X,S,p)$ be a supported action. Define $X_{e} = p^{-1}(e)$.
From now on supported actions of a semilattice $E$ will be termed a presheaf
over $E$ and denoted by $X = (X,E,p)$.

\begin{lem} \label[lemma]{lem: subpresheaf}
	Let $(X,E,p)$ be a presheaf. Let $Y \subseteq X$. Then,
		$Y$ is closed under the action of $E$ if and only if 
		$Y$ is an order-ideal in $X$.
\end{lem}

\begin{proof}
	One implication is a consequence of \Cref{subact ideal}. We show the other implication.
Suppose $Y$ is an order-ideal of $X$.
Let $y \in Y$ and $e \in E$. Then, 
	$y \cdot e \leq y$ and so, by definition, $y \cdot e \in Y$ and $Y$ is closed under the action.
\end{proof}

Let $(X,E,p)$ be a presheaf and let $Y$ be an order-ideal of $X$.
By \Cref{lem: subpresheaf} the action $X \times E \to X$ restricts to an action $Y \times E \to Y$.
Denote the restriction of $p$ to $Y$ by $p|_{Y}\colon Y \to p(Y)$.
Observe that for the order-ideal $Y$ the image $p(Y)$ is a subsemilattice of $E$.
Let $y_1,y_2 \in Y$ and let $e = p(y_1)p(y_2)$. Put $y' = y_1 \cdot e$. Then,
$p(y') = p(y_1)e = e$ and $y' \in Y$ since $Y$ is an order-ideal.
Thus a \emph{subpresheaf} of $X$ consists of an order-ideal $Y \subseteq X$, a subsemilattice $F \subseteq E$ containing $p(Y)$,
and a map $q\colon Y \to F$ such that $q(y) = p(y)$.
Then, $(Y, F, q)$ is a presheaf which we denote by $Y$. Note that a subpresheaf is a subaction, as it is closed under the
action, however the semilattice is also restricted. 
A particularly important class of subpresheaves are the sets
$X \cdot e = \{x \cdot e: x \in X\}$ for all $e \in E(S)$.
Observe that 
\[
	X \cdot e = \{x \in X : p(x) \leq e\}.
\]
Hence, $X \cdot e$ is an order-ideal of $X$ and $p(X \cdot e) \subseteq e^{\downarrow}$.
We denote the restriction of $p$ to $X \cdot e$ by $p_{e}\colon X \cdot e \to e^{\downarrow}$ and call
each of the subpresheaves $(X \cdot e, e^{\downarrow}, p_{e})$ a \emph{principal subpresheaf of $X$}.
Note that if $X$ has global support then all the principal subpresheaves of $X$ have global support.
We will usually denote a subpresheaf $(Y,F,q)$ by $(Y,F)$.
\begin{rem}
	The usual definition of a subpresheaf of a presheaf $(X_{e}, \phi^{e}_{f})$ over $E$, as given in \cite{sheaves}, is 
	a presheaf  $(Y_{e}, \varphi^{e}_{f})$ such that
	 $Y_{e} \subseteq X_{e}$ for each $e \in E$ and
 for any $f \leq e$ then $\varphi^{f}_{e}$ is the restriction of $\phi^{f}_{e}$ to $Y_{e}$.
By \Cref{NoFunc} we see that this notion is equivalent to our order-theoretic one.
\end{rem}

We are interested in the isomorphisms between all subpresheaves and wish to prove that they form an inverse semigroup.

\begin{lem}
	Let $(X,E,p)$ be a presheaf.
	Let $(Y,F)$ and $(Y',F')$ be two subpresheaves of $X$ and
	let $\alpha\colon Y \to Y'$ and $\theta\colon F \to F'$ be a pair functions.
	Then the following are equivalent:
	\begin{enumerate}[1.]
		\item 
		$(\alpha, \theta)$ is an isomorphism of supported actions,
		\item 
		$\alpha$ and $\theta$ are order-isomorphisms such that $p(\alpha(y)) = \theta(p(y))$ for all $y \in Y$.
	\end{enumerate}
\end{lem}

\begin{proof}
	Both statements imply that $\alpha$ and $\theta $ are bijections. By definition $\theta $ is an order-isomorphism
	if and only if $\theta $ is an isomorphism.
	Suppose that $(\alpha, \theta) $ is an isomorphism of supported actions. By definition
	$p(\alpha(y)) = \theta(p(y))$ for all $y \in Y$.
	Let $x \leq y $ in $Y$. Then,
	\[
		\alpha(x) = \alpha(y \cdot p(x)) = \alpha(y) \cdot \theta(p(x)) = \alpha(y) \cdot p(\alpha(x))
	\]
	which is the definition of $\alpha(x) \leq \alpha(y)$.
	Conversely,
	suppose that $\alpha$ is order-preserving and
	let $y \in Y$ and $e \in F$. Then, $y \cdot e \leq y$. Hence,
	$\alpha(y \cdot e) \leq \alpha(y)$. By definition of the partial order and the assumption $p(\alpha(y)) = \theta(p(y))$
	\[
		\alpha(y \cdot e) = \alpha(y) \cdot p(\alpha(y \cdot e)) = \alpha(y) \cdot \theta(p(y \cdot e)) =
		\alpha(y) \cdot \theta(p(y)) \theta(e) = \alpha(y) \cdot \theta(e)
	.\]
\end{proof}

Let $(X,E,p)$ be a presheaf.
Define the \emph{partial automorphisms} of $X$ to be the set of isomorphisms between
subpresheaves of $X$. 
The partial automorphisms of $X$ then forms a groupoid,
hence we can use the Ehresmann-Schein-Nambooripad Theorem to prove
that it is an inverse semigroup. We recall the following theorem, further details can be found in \cite[Chap. 4]{LawsonBook}.
\begin{thm}[{\cite[Theorem 4.1.8]{LawsonBook}}] \label{ESN}
	The category of inverse semigroups and homomorphisms is isomorphic to the category of inductive groupoids and inductive
	functors.
\end{thm}

We now prove that the set of partial automorphisms of a presheaf forms an inductive groupoid and then apply the previous theorem
to show it forms an inverse semigroup.

\begin{lem} \label[lemma]{partial auto}
	The set of partial automorphisms of a presheaf forms an inverse semigroup.
\end{lem}

\begin{proof}
Let $(X,E,p)$ be a presheaf. The subpresheaves of $X$ are partially ordered with respect to inclusion,  $ (Z,H) \leq (Y,F) $ if and only if $Z\subseteq Y$ and $ H\subseteq F$. This induces 
a partial order on the 
groupoid of partial automorphisms of $(X,E,p)$
given by $(\alpha, \theta) \leq (\alpha', \theta')$ whenever
the domain and range of $(\alpha, \theta)$ are contained in the domain and range of $(\alpha', \theta') $.
Let $(\alpha, \theta)\colon (Y,F) \to (Y',F')$ be a partial automorphism and $(Z,H) \leq (Y,F)$.
Then, since $\alpha$ and $\theta$ are order-isomorphisms, $\alpha(Z)$ is an order-ideal and $\theta(H)$ is a subsemilattice.
Since $p(\alpha(y)) = \theta(p(y))$ for all $y \in Y$
then 
$p(\alpha(Z)) = \theta(p(Z)) \subseteq \theta(H)$.
Hence, the restriction to $(\alpha|_{Z}, \theta|_{H}): (Z,H) \to (\alpha(Z), \theta(H))$ is a partial automorphism of $X$. We have proved that
the set of partial automorphisms forms an ordered groupoid. 
Since the intersection of order-ideals is an order-ideal it follows that for any two subpresheaves  $(Z,H) $ and $(Y,F)$ their intersection $(Z \cap Y, H \cap F)$ is a subpresheaf and 
the set of identities forms a semilattice. Therefore,
the set of partial automorphisms forms an inductive groupoid and so,
by \Cref{ESN},  an inverse semigroup.
\end{proof}

\section{A generalisation of the Munn semigroup}
We shall now generalise the classical Munn semigroup $T_{E}$ by replacing the semilattice $E$ by a presheaf $X$ over $E$. To do this instead of taking the inverse semigroup of order-isomorphisms
between principal order-ideals, we take the inverse semigroup of isomorphisms between principal subpresheaves. Recall that for a presheaf $(X,E,p)$ the principal subpresheaves are of the form $(X \cdot e, e^{\downarrow},
p_{e})$ where $X \cdot e = \{x \in X: p(x) \leq e\}$ and $p_{e} = p|_{X\cdot e}$, for each $e \in E$. An isomorphism 
between two principal subpresheaves of $X$ is a pair $(\alpha, \theta)$ consisting of order-isomorphisms such that 
the following diagram commutes
\begin{equation} \label{prince}
	\begin{tikzcd}
		X \cdot e \ar[d, "p_{e}"'] \ar[r, "\alpha"] & X \cdot f \ar[d, "p_{f}" ] \\
		e^{\downarrow} \ar[r, "\theta"] & f^{\downarrow}
	\end{tikzcd}
\end{equation}
Define $T_{X}$ to be the set of isomorphisms between the principal subpresheaves of 
$(X,E,p)$.

\begin{thm} \label{prince algebra}
	Let $(X,E,p)$ be a presheaf. Then, $T_{X}$ is an inverse semigroup.
\end{thm}

\begin{proof}
	We prove that $T_{X}$ is an inverse subsemigroup of the partial automorphisms of $X$. By \Cref{partial auto}
	it is enough to show that $T_{X}$ is compatible with restriction and that the identities from a 
	semilattice.
	Let $(\alpha, \theta)\colon (X \cdot e, e^{\downarrow}) \to (X \cdot f, f^{\downarrow})$ be in $T_{X}$ and $(X \cdot h, h^{\downarrow}) \leq (X \cdot e, f^{\downarrow})$.
	This implies that $h \leq e$ so, as $\theta $ is order-preserving, $\theta(h) \leq \theta(e)=f$.
	Since $(\alpha, \theta)$ is
	an isomorphism of supported actions then $\alpha(x \cdot h) = \alpha(x) \cdot \theta(h)$ for all $x \in X \cdot e$ and
	hence
	\[
		\alpha(X \cdot h) = \alpha(X \cdot eh) = \alpha(X \cdot e) \cdot \theta(h) = X \cdot f \theta(h) = X \cdot \theta(h)
	.\]
	So every element of $T_{X}$ can be restricted to a smaller principal subpresheaf. It is now enough to show that
	the identities form a semilattice. Let $(X \cdot e, e^{\downarrow})$ and $(X \cdot f, f^{\downarrow})$ be two
	principal subpresheaves, we prove they are closed under intersections
	\[
		(X \cdot e, e^{\downarrow})  \cap (X \cdot f, f^{\downarrow}) = (X \cdot ef, (ef)^{\downarrow}).
	\]
	Observe that $(ef)^{\downarrow} = e^{\downarrow } \cap  f^{\downarrow}$. We show that $X \cdot ef = X \cdot e \cap X\cdot f$.
 Let $x \in X \cdot ef$. Then,  
$p(x) \leq ef$ so $p(x) \leq e$ and $p(x) \leq f$. Hence, $x \in X \cdot e$ and $x \in X \cdot f$.
Conversely, let $x \in X \cdot e \cap X \cdot f$. Then, $p(x) \leq p(x)p(x) \leq ef$ and therefore
$X \cdot e \cap X \cdot f = X \cdot ef$.
	 Therefore, $T_{X}$ is an inductive groupoid and hence an inverse semigroup.
\end{proof}
We call $T_{X}$ the \emph{generalised Munn semigroup} of $(X,E,p)$. 
Denote elements of $T_{X}$ by pairs $(\alpha, \theta_{\alpha })$.

\begin{lem} \label[lemma]{lem: theta}
	Let $(X,E,p)$ be a presheaf. Then:
	\begin{enumerate}
		\item 
	The projection $\gamma\colon T_{X} \to T_{E}$ defined by $(\alpha, \theta_{\alpha}) \mapsto \theta_{\alpha}$
	is an idempotent-separating homomorphism and in
	particular 
	$E(T_{X}) \cong E$.
	\item 
If $p$ is surjective
the projection $T_{X} \to \mathcal{I}(X) 
$ defined by $(\alpha, \theta) \mapsto \alpha$ is injective.
		
	\end{enumerate}
\end{lem}

\begin{proof}
\qquad

	\begin{enumerate}
\item 
Let $(\alpha,\theta_{\alpha}), (\beta, \theta_{\beta}) \in T_{X}$. Then,
\[
	(\alpha,\theta_{\alpha})(\beta, \theta_{\beta}) = (\alpha \beta, \theta_{\alpha } \theta_{\beta})
.\]
Hence, $\theta_{\alpha }\theta_{\beta } = \theta_{\alpha \beta}$ and $\gamma $ is a homomorphism.
The idempotents of $T_{X}$ are
$(\id_{X \cdot e}, \id_{e^{\downarrow}}) \in T_{X}$ for each $e \in E$. Since the idempotents of $T_{E}$ are
$\id_{e^{\downarrow}}$ for each $e \in E$ then $\gamma $ is idempotent-separating and $\gamma(E(T_{X})) = E(T_{E}) \cong E$.
\item 
Let $(\alpha, \theta)\colon (X \cdot e, e^{\downarrow}) \to (X \cdot f, f^{\downarrow})$ as in in diagram (\ref{prince}).
Let $h \leq e$ then, since $p$ is surjective,  $p_{e}$ is surjective and
there exists an $x \in X \cdot e$ such that $p_{e}(x) = h$ and
$\theta(h) = p_{f}(\alpha(x))$.
Therefore, if $(\alpha, \theta)$ exists then
 $\theta$ is completely determined by $\alpha$ and the projection is injective.

	\end{enumerate}
\end{proof}

 We check that $T_{X}$ really is a generalisation of the 
Munn semigroup.

\begin{prop} \label[proposition]{Munn action}
	Let $E$ be a semilattice. Put $X = E$ and define $p\colon X \to  E$ to be the identity map. Then,
	$X$ is a presheaf over $E$ and $T_{X} \cong T_{E}$.
\end{prop}

\begin{proof}
	It is a simple check that $E$ acting on itself by right multiplication makes $(X,E,p)$ a globally supported 
	action. Note that the principal subpresheaves of $X$ are the order-ideals $E \cdot e = e^{\downarrow}$ for
	all $e \in E$.
Suppose there is an element $(\alpha, \theta_{\alpha }) \in T_{X }$ where 
$\alpha\colon  e^{\downarrow} \to  f^{\downarrow}$ and $\theta_{\alpha }\colon e^{\downarrow } \to  f^{\downarrow}$. Then,
for any $y \in  e^{\downarrow}$,
\[
\alpha(y) = p(\alpha(y)) = \theta_{\alpha}(p(y)) = \theta_{\alpha}(y)
.\] Therefore, 
\[
	T_{X} = \{ (\theta: e^{\downarrow} \to f^{\downarrow}, \theta: e^{\downarrow} \to f^{\downarrow}) : \theta \in T_{E}\} \cong T_{E}
.\]
\end{proof}

We generalise the Munn representation theorem such that every supported action of an inverse semigroup gives rise to a representation
on to the generalised Munn semigroup.

\begin{thm}[Generalised Munn representation] \label{Thm: General Munn}
	Let $(X,S,p)$ be a supported action and let $Y=(X,E(S),p)$ be the presheaf obtained by restricting the action $(X,S,p)$ 
	to the idempotents.
	Then, there is an idempotent-separating 
	representation $\xi\colon S \to T_{Y}$
	 whose image is a wide inverse subsemigroup of $T_{Y}$.
\end{thm}

\begin{proof}
	 For each $s \in S$ we define the function $\alpha_{s}\colon X \cdot \mathbf{r}(s) \to X \cdot \mathbf{d}(s)$ by
	 $\alpha_{s}(y) = y \cdot s$ . Since
	$p(y \cdot s) = s^{-1}p(y)s$ then $p(y \cdot s) \leq \mathbf{d}(s)$ and the function is well-defined.
	 For $y \in X \cdot \mathbf{r}(s)$, since
	$p(y) \leq \mathbf{r}(s)$, then
	\[
		(y \cdot s) \cdot s^{-1} = y \cdot ss^{-1} = y \cdot p(y)ss^{-1} = y \cdot p(y) = y.
	\]
	Hence there is a bijection $X \cdot \mathbf{r}(s) \to X \cdot \mathbf{d}(s)$. By
	\Cref{steinlem} the action is compatible with the partial order so the map is an order-isomorphism.
	By the Munn representation, \Cref{classical munn}, for each $s \in S$ there is
	an order-isomorphism 
	$\theta_s\colon (ss^{-1})^{\downarrow} \to (s^{-1}s)^{\downarrow}$ given by
	$\theta_{s}(e) = s^{-1}es$.
	Let $y \in X \cdot \mathbf{r}(s)$ and $e \in E(S)$. Then,
	\[
		\alpha_{s}(y \cdot e) = y \cdot ss^{-1}es = \alpha_{s}(y) \cdot \theta_{s}(e).
	\]
	And,
\[
p(\alpha_s(y)) = p(y \cdot s) = s^{-1}p(y)s = \theta_s(p(y)).
\] 
Therefore, $(\alpha_{s}, \theta_{s})$ is a well-defined element of $T_Y$.
Define the function $\xi\colon S \to T_Y$ by
 $s \mapsto (\alpha_{s^{-1}}, \theta_{s^{-1}})$. 
 We check that $\xi $ is a homomorphism. Let $s,t \in S$. Note that $(\alpha_{t^{-1}})^{-1}=\alpha_{t}$ and so
 \[
 \dom (\alpha_{s^{-1}} \alpha_{ t^{-1} }) = \alpha_{t }( X \cdot \mathbf{r}(t) \cap X \cdot \mathbf{d}(s))
	 = X \cdot t^{-1}\mathbf{r}(t)\mathbf{d}(s)t = X \cdot \mathbf{d}(st).
 \]
Similarly $\ima \alpha_{s^{-1}} \alpha_{t^{-1}} = X \cdot \mathbf{r}(st)$.
Hence, by (SA1), $\alpha_{(st)^{-1} } = \alpha_{s^{-1} } \alpha_{t^{-1}}$ and $\xi $ is a homomorphism.
We show that this is
a wide representation. An idempotent of $T_{Y}$ is of the form $(\id_{X \cdot e}, \id_{e^{\downarrow}})$ for some $e \in E(S)$.
By \Cref{classical munn}, $\id_{e^{\downarrow}} = \theta_{e}$.
For all
$x \in X \cdot e$, since $p(x) \leq e$, \[
\alpha_e(x) = x \cdot e = x \cdot p(x)e = x \cdot p(x) = x
.\]
Hence, $\id_{X \cdot e } = \alpha_{e }$  and $\xi(e) = (\id_{X \cdot e}, \id_{e^{\downarrow}})$.
Lastly note that $\xi$ is idempotent-separating since if 
$(\alpha_{e},\theta_{e}) = (\alpha_{f}, \theta_{f})
	.$ Then,
	\[
		e = \id_{e^{\downarrow}}(e) =  \theta_e(e) = \theta_{f}(e) = ef
	\]
	and with a symmetric argument
	$f = ef$. Therefore, $e=f$ and the representation is injective on the idempotents.

\end{proof}

We now describe the kernel of $\xi$, the homomorphism defined in \Cref{Thm: General Munn},
in particular when the kernel is equality and hence the representation
is faithful. Let $(X,S,p)$ be a supported action.
Define the relation $\rho_{X}$ on $S$ by
\[
	(s,t) \in \rho_{X} \iff x \cdot s = x\cdot t \ \mbox{ for all } \ x \in X
.\]

\begin{prop} \label[lemma]{IdempRho}
Let $(X,S,p)$ be a supported action. Then:
\begin{enumerate}
\item 
	 $\rho_{X}$ is a congruence on $S$ and $\ker \xi \subseteq \rho_{X}$.
\item 
	$\rho_{X}$ is idempotent-separating if and only if $\ker \xi = \rho_{X}$.
\item 
	$\xi\colon S \to T_{Y}$ is faithful if $\rho_{X}$ is equality.
\item 
If $p$ is surjective then $\rho_{X}$ is idempotent-separating.
\end{enumerate}

\end{prop}

\begin{proof}
	\quad 

	\begin{enumerate}
	\item 
		Observe that
	$\rho_{X}$ is an equivalence  relation, we show that it is a congruence.
  Let $(s,t), (s',t') \in \rho_{X}$. Then,
	\[
		(x \cdot s) \cdot s' = (x \cdot t) \cdot s' = (x \cdot t) \cdot t' \quad \mbox{for all }x \in X 
	\]
	 implying $(ss', tt') \in \rho_{X}$ and $\rho_{X}$ is a congruence.
Let $(s,t) \in \ker \xi $ then $(s^{-1}, t^{-1}) \in \ker \xi $ which implies that
$(\alpha_{s}, \theta_{s}) = (\alpha_{t}, \theta_{t})$. Since $\ker \xi$ is an idempotent-separating congruence then
$ss^{-1} = tt^{-1}$ and 
\[
	x \cdot s = x \cdot (ss^{-1})s =  \alpha_{s}(x \cdot ss^{-1}) = \alpha_{t}(x \cdot tt^{-1}) = x \cdot t
\]
for all $x \in X$. Therefore, $\ker \xi \subseteq \rho_{X}$.
\item 
Since $\xi$ is idempotent-separating if $\ker \xi = \rho_{X}$ then $\rho_{X}$ would be idempotent separating.
We prove the converse. Suppose $\rho_{X}$ is idempotent-separating.
Let $(s,t) \in \rho_{X}$. Since $\rho_{X}$ is a congruence $(ss^{-1}, tt^{-1}) \in \rho_{X}$ and therefore
$X \cdot ss^{-1} = X \cdot tt^{-1}$. For all $x \in X \cdot \mathbf{r}(s)$ 
$\alpha_{s}(x) = x \cdot s = x \cdot t = \alpha_{t}(x)$ so $\alpha_{s} = \alpha_{t}$.
Since $\rho_{X}$ is idempotent-separating $\rho_{X} \subseteq \mu $ and hence $\theta_{s} = \theta_{t}$.
Therefore,
$(\alpha_{s}, \theta_{s}) = (\alpha_{t}, \theta_{t})$ implying 
$(s^{-1}, t^{-1}) \in \ker \xi $ and so $(s,t) \in \ker \xi$.
\item 
If $\rho_{X}$ is equality they by part (1) $\ker \xi $ is equality. Therefore, $\xi $ is an injective homomorphism.

\item 
	Let $(e,f) \in \rho_{X}$ where $e,f$ are idempotents. Then $x \cdot e = x \cdot f$ for all  $x \in X
	$. 
	Applying $p$,
	\[
	p(x)e = p(x)f \quad \mbox{for all } x \in X
	.\] 
	Since the action has global support
	there exists $y \in X$ such that $p(y) = e$. Then,
	\[
	e = p(y)e = p(y)f =ef \mbox{ and } e \leq f
	.\] By symmetry $f \leq e$ and so $e = f$.
	\end{enumerate}
\end{proof}

We call $\rho_{X}$ the 
\emph{characteristic congruence} of $X$.

\begin{example}
	We give an example of a supported action where the characteristic congruence is idempotent-separating but for which
	the action is without global support. Let $E = \{1,e, f,0\}$ be a semilattice such that  $1$ is the top element,
	$0$ is the bottom element and  $ef = 0$. Let 
	$X = \{e_1,f_1,0\}$ be a set, and $p\colon X \to E$ a map defined by $p(e_1) = e$, $p(f_1) = f$  and $p(0) = 0$.
	We define the action by $x \cdot 1 = x$ and $x \cdot 0 = 0$ for all $x \in X$.
	Clearly $p$ is not surjective, however $\rho_{X}$ is also equality and therefore idempotent-separating.
\end{example}
So far, we have shown that every supported action comes with a characteristic congruence.
We now prove that every idempotent-separating congruence on $S$ arises as the characteristic congruence of some globally 
supported action
of $S$.

\begin{prop} \label[proposition]{thm: any congruence}
	Let $S$ be an inverse semigroup and let $\rho $ be an
 idempotent-separating congruence on $S$.
 Define $S/\rho \times S \to S/\rho$ by $\rho(s) \cdot t = \rho(st)$ and $p\colon S/\rho \to E(S) $ by
 $p(\rho(s)) = \mathbf{d}(s)$.
	Then $(S/\rho,S,p)$ is a globally supported action with characteristic congruence 
	$\rho$.
\end{prop}

\begin{proof}
	First we prove that this is a supported action. 
	Suppose $\rho(s)=\rho(t)$ then
	 $\rho(s^{-1}s) = \rho(t^{-1}t)$. Since $\rho$ is idempotent-separating then $s^{-1}s = t^{-1}t$
	and the function $p$ is well-defined.
The action is clearly associative, and it is well-defined since  $\rho(s_1) = \rho(s_2)$ implies 
	 $\rho(s_1t) = \rho(s_2t)$ for any $t \in S$, by $\rho$ being a congruence.
	 Note that
	\[
		p(\rho(s)\cdot t) = p(\rho(st)) = t^{-1}p(\rho(s))t
	\]
	and \[
	\rho(s) \cdot p(\rho(s)) = \rho(s) \cdot s^{-1}s  = \rho(ss^{-1}s)= \rho(s).
\]
Hence, the action is a supported action. Moreover, since $\rho $ it idempotent-seperating, the action is global.
	Suppose that $\rho(a) \cdot s = \rho(a) \cdot t $ for all $a \in S$. Since characteristic congruences are 
	idempotent separating then $ss^{-1}=tt^{-1}$ and
	\[
		\rho(s) = \rho(ss^{-1}) \cdot s =  \rho(tt^{-1}) \cdot t = \rho(t)
	.\] 
	Therefore, $(s,t) \in \rho$.
	Conversely, suppose $(s,t) \in \rho$  then we know for all $a \in S$ that $(as,at)  \in \rho 
	.$ 
	Hence, \[
	\rho(a) \cdot s = \rho(as) = \rho(at) = \rho(a)\cdot t
	\] 
for all $a \in S$. We have shown that the characteristic congruence is $\rho$.
\end{proof}

\begin{example} \label[example]{parti point}
We define two non-isomorphic globally supported actions with the same characteristic congruence. Hence,
there is no bijection between the globally supported actions of $S$ and idempotent-separating congruences on $S$.
Let $X = \{0,1,2\}$ have the finite particular point topology. That is $X$ has open sets
$\{\emptyset, \{0\} , \{0,1\}, \{0,2\}, X\}$. Let $S$ be the inverse semigroup of partial homeomorphisms of $X$.
Then, $S$ has two non-idempotent elements $f\colon \{0,1\} \to \{0,2\}$, such that $f(0) = 0$ and $f(1) = 2$, and its inverse.
Since $f^{-1}\id_{ \{0,2\}}f = \id_{ \{0,1\}}$, by definition of $\mu$, $S$ is fundamental. Therefore,
the characteristic congruence of $(S,S,\mathbf{d})$ and $(E(S),S,\mathbf{d})$ are both equality.
\end{example}

A special case of \Cref{thm: any congruence} is when $\rho$ is equality. Then, $(S/ \rho, S,p) = (S,S, \mathbf{d})$ which was the
supported action defined in Example \ref{right ideal}. Since the action is globally supported then by
\Cref{IdempRho} $\ker \xi = \rho$ and therefore $\xi\colon S \to T_{S}$ is an embedding into the generalised Munn semigroup.
This result is similar to one of Reilly \cite{reilly1977enlarging} who constructed an embedding of every inverse semigroup
into one with similar structure to the Munn semigroup.

\begin{cor}
	Let $S$ be an inverse semigroup. Then, $(S,S,\mathbf{d})$ is the supported action corresponding to $S $ acting on itself by composition on the right 
 and $S$  embeds into the generalised Munn semigroup
	$T_{S}$.
\end{cor}

Let $S$ be an inverse semigroup.
By the previous result we can see that the supported action $(S,S,\mathbf{d})$ corresponds to
the usual Wagner-Preston representation of inverse semigroups.
Recall from Example \ref{clifford subaction} the supported action $(E(S), S, \id_{E(S)})$. By
\Cref{Munn action} the generalised Munn semigroup of $E(S)$ is isomorphic to the Munn semigroup $T_{E(S)}$. Moreover,
it is easy to see that the characteristic congruence $\rho_{E(S)}$ is $\mu$. Hence,  $(E(S),S,\id_{E(S)})$ corresponds to
the Munn representation theorem.

We now show that every supported action gives rise to an idempotent-separating 
homomorphism on to the generalised Munn semigroup.

\begin{thm} \label{actions are homo}
	Let $S$ be an inverse semigroup and $(X,E(S),p)$ a presheaf. Then, the following are equivalent:
	\begin{enumerate}
	\item There exists an idempotent-separating homomorphism $\phi\colon S \to T_{X}$ whose image is a 
	wide inverse subsemigroup of $T_{X}$.
	\item 
	There exists a supported action $(X,S,\overline{p})$ such that the presheaf obtained by restricting 
	to the idempotents is isomorphic to $(X,E(S),p)$.
	\end{enumerate}
\end{thm}

\begin{proof}
	By \Cref{Thm: General Munn} for every supported action $(X,S,p)$ there is an idempotent-separating
	homomorphism $\xi\colon S \to T_{Y}$, whose image is a wide inverse subsemigroup of $T_{Y}$,
	where $Y$ is the presheaf obtained from restricting $(X,S,p)$ to the idempotents.
	We prove the converse.
	Let $\phi\colon S \to T_{X}$ be an idempotent-separating homomorphism such that $\phi(E(S)) = E(T_{X})$.
	Then, the restriction 
	$\phi|_{E(S)}\colon E(S) \to E(T_{X})$ is an isomorphism.
	Define $\overline{p}\colon X \to E(S)$ by
	\[
		\overline{p}(x) = \phi|_{E(S)}^{-1}(\id_{X \cdot p(x)}, \id_{p(x)^{\downarrow}}).
	\]
	Define the action $X \times S \to X$  by
	 \[
		 x \circ  s = \alpha_{s}^{-1}(x \cdot f).
	 \] where $\phi(s) = (\alpha_{s}, \theta_{s})\colon (X \cdot e, e^{\downarrow}) \to (X \cdot f, f^{\downarrow})$.
	 We show that $(X,S, \overline{p})$ is a supported action. If $\phi(s)$ is as before and
	 $\phi(t) = (\alpha_{t}, \theta_{t})\colon (X \cdot g, g^{\downarrow}) \to (X \cdot h, h^{\downarrow})$ then:
	 \[
		 (x \circ s) \circ t = \alpha_{t}^{-1}(\alpha_{s}^{-1}(x \cdot f) \cdot h) 
		 = \alpha_{t}^{-1}\alpha_{s}^{-1}(x \cdot  \theta_{s}(eh)).
	 \]
	 Since
	 \[
		 \ima \alpha_{s}\alpha_{t} = \alpha_{s}( X \cdot e \cap X \cdot h)  
		 = \alpha_{s}(X \cdot eh) = X \cdot \theta_{s}(eh)
	 \]
	 it is shown that
	  $(x \circ s) \circ t = x \circ st$ and (SA1) holds. For all $x \in X$
	  \[
		  x \circ \overline{p}(x) = x \circ \phi|_{E(S)}^{-1}(\id_{X \cdot p(x)}, \id_{p(x)^{\downarrow}})
		  = \id_{X \cdot p(x)}(x) = x.
	  \]
	  Hence, (SA2) holds. Let $\phi(s) = (\alpha_{s}, \theta_{s})$ as before. Since
	  \[
		  p(\alpha^{-1}(x \cdot f)) = \theta_{\alpha}^{-1}(p(x)f) = e \theta_{\alpha}^{-1}(p(x))
	  \]
	  and
\[
	\alpha_{s}^{-1}(X \cdot p(x) \cap  X \cdot f) = \alpha_{s}^{-1}(X \cdot fp(x)) = \alpha_{s}^{-1}(X \cdot f)
	\cdot \theta_{s}^{-1}(p(x)) = X \cdot e \theta^{-1}_{s}(p(x))
.\] 
Then,
\[
	\phi(s^{-1}\overline{p}(x)s) = 
	(\alpha_{s}^{-1},\theta_{s}^{-1})(\id_{X \cdot p(x)}, \id_{p(x)^{\downarrow}})(\alpha_{s},\theta_{s})
	= \phi\overline{p}(x \circ s).
\]
Therefore, since $\phi|_{E(S)}\colon E(S) \to E(T_{X}) $ is an isomorphism, (SA3) holds.
Note that the restriction of $(X,S,\overline{p})$ to idempotents is a presheaf isomorphic to $(X,E(S),p)$ by the isomorphism
$E(S) \to E(S)$ defined by $e \mapsto \phi|_{E(S)}^{-1}(\id_{X \cdot e}, \id_{e^{\downarrow}}) $.
\end{proof}

\begin{example}
	In the case of the Munn semigroup two semilattices $E$ and $F$ are isomorphic if and only $T_{E}$ and $T_{F}$ 
	are isomorphic. We show that the same is not true for the generalised case.
	Let $E = \{e,f, 0\}$ be the semilattice with $0 \leq f \leq e$.
	Let  $X = \{e_1,e_2,e_3,f_1,f_2,0\}$, with slight abuse of notation in using $0$ again,
	and $p\colon X \to E$ be defined by $p(e_{i}) = e$ for $i \in \{1,2,3\}$,
	$p(f_j) = f$ for $j \in \{1,2\}$, $p(0) = 0$.
	Then, $(X,E,p)$ is a presheaf with $e_{i} \cdot f = f_1$ for $i \in \{1,2\}$, $e_3 \cdot f = f_2$, and
	$x \cdot 0 = 0$ for all $x \in X$.
	Note that $X \cdot f = \{f_1,f_2,0\}$ and $X \cdot 0 = \{0\}$.
	Let $(\alpha, \theta_{\alpha}) \in T_{X}$ such that $\alpha\colon X \to X$.
	Suppose that $\alpha(e_1) = e_3$. Then, since $\alpha$ is an order-isomorphism $\alpha(f_1) \leq \alpha(e_1) = e_{3}$.
	Hence,
	$\alpha(f_1) = f_{2}$ the only element in $X_{f}$ below $e_3$ in the partial order of $X$.
	Since $\alpha(e_2) \geq \alpha(f_1) = f_2$ it implies $\alpha(e_2) = e_{3}$. This is a 
	contradiction to $\alpha$ being an isomorphism. So $\alpha(e_1) = e_2$ or $e_1$ only.
	Hence, the only possible $(\alpha, \theta_{\alpha}) \in T_{X}$ such that $\alpha\colon X \to X$ are
	$(\id_{X}, \id_{E})$ and $(\beta, \theta_{\beta})$ defined by swapping $e_1 $ and $e_2$, and fixing all other $x \in X$.
	Similarly, the only $(\alpha, \theta_{\alpha}) \in T_{X}$ such that $\alpha\colon X \cdot f \to X\cdot f$ are
	$(\id_{X \cdot f}, \id_{f^{\downarrow}})$ and the map swapping $f_1$ and $f_2$.
	Let $Y = \{e_1,e_2,f_1,f_2,0\}$ be the subpresheaf of $X$. Then, it is easily verified that
	$T_{X} \cong T_{Y}$.
	Let $S$ be an inverse semigroup with idempotents isomorphic to $E$.
	By \Cref{actions are homo} since $T_{X}$ and $T_{Y}$ are isomorphic there is a bijection between the
	supported actions of $S$ on $X$ with presheaf isomorphic to $(X,E,p)$ and supported actions on $Y$
	with presheaf isomorphic to $(Y,E, p)$.
\end{example}

\section{Connection with the work of Zhitomirskiy}

We now show the connections of our work with topology. We will start with the case of the Munn semigroup and then
continue with the generalised case.
Let $(X, \tau)$ be a topological space. Denote by $\mathcal{I}(X, \tau)$ the inverse semigroup of all
homeomorphisms between the open subsets of $X$. 
We recall some basic theory of frames and their related spaces, details may be found in \cite{sheaves}.
The \emph{join} of elements in a partially ordered set is the least upper bound.
A \emph{frame} is a semilattice in which all arbitrary joins exist and is infinitely
distributive. A \emph{frame homomorphism} is a homomorphism
preserving arbitrary joins.
If $X$ is a topological space, denote by $\Omega(X) $ the frame of
all open sets of $X$.
Observe that this construction $X \mapsto \Omega(X)$ can be made into a functor as for any continuous map 
$\theta\colon X \to Y$, the 
inverse image $\theta^{-1}\colon \Omega(Y) \to \Omega(X)$ defines a frame homomorphism. We will also denote
$\theta^{-1}$ by $\Omega(\theta)$.
From every frame $E$ we construct a space $\pt (E)$ by taking the set of completely prime filters and defining a basis
of the topology with sets $\ats_{e} = \{ F \in \pt (E) : e \in F \}$ for all $e \in E$. This construction also defines
a contravariant functor since the inverse image of a frame homomorphism $\theta$, denoted $\pt(\theta)$,
maps completely prime filters to completely prime filters.
Let $X$ be a topological space and $E$ be a frame.
Define $\eta_{X}\colon X \to \pt \Omega(X)$ to be the map taking each point to its neighbourhood filter and
$\epsilon_{E}\colon \Omega \pt(E) \to E$ to be the morphism in $\catname{Frame}^{\op}$ representing
	the frame homomorphism $E \to \Omega \pt(E)$ taking
	$e \mapsto \ats_{e}$.
Then, $\eta$ is the unit and $\epsilon$ is the counit defining an adjunction between $\Omega$ and $\pt$. 

\begin{thm}[{\cite[Chap. 9]{sheaves}}] \label{adjunction}
	The functor $\Omega\colon \catname{Spaces} \to \catname{Frame}^{\op}$ taking $X \to \Omega(X)$ has a right adjoint
	$\pt\colon \mbox{\normalfont{Frame}}^{\op} \to \catname{Spaces}$.

\end{thm}

	A space $X$ is said to be \emph{sober} if $\eta_{X}\colon X \to \pt \Omega(X)$ is a homeomorphism. Explicitly, $X$ is sober
	if every completely prime filter in $\Omega(X)$ is the neighbourhood filter of a unique point in $X$.
	A frame $E$ is said to have \emph{enough points} if $\epsilon_{E}\colon \Omega \pt(E) \to E$ is an isomorphism.
\begin{lem} \label[lemma]{lem: image}
	Let $X$ be a sober space and $\theta \in T_{\Omega(X)}$. Then, if 
	$\theta\colon \Omega(V) \to \Omega(U)$, the map
	\[
		\theta^{\ast } =  \eta_{V}^{-1}   \pt(\theta)    \eta_{U}\colon U \to V
	\]
	is a homeomorphism.
	Moreover, $\theta^{\ast}(W) = \theta^{-1}(W)$
	for all open subsets $W$ of $U$.
\end{lem}

\begin{proof}
	By being a functor $\pt(\theta)\colon \pt \Omega(U) \to \pt \Omega(V)$ is a homeomorphism.
	Since open sets of a sober space are sober with respect to the subspace topology then $\eta_{U}$ and $\eta_{V}$
	are homeomorphisms. Hence, $\theta^{\ast}$ is a homeomorphism.
	Following the definitions $\eta_{U}(W)$ is the set
	of completely prime filters of $\Omega(U)$ containing $W$.
	By $\pt(\theta) $ this is mapped to the set of
	completely prime filters containing $\theta^{-1}(W)$. Therefore, $\theta^{\ast}(W)$ is all points
	contained in $\theta^{-1}(W)$.
\end{proof}

\begin{lem} \label[lemma]{lem: im gives equal}
	Let $U,V$ be sober spaces and $f,g\colon U \to V$ be a pair of homeomorphisms. 
	If
	 $f(A) = g(A)$ for all $A \in \Omega(U)$ then $f = g$.
\end{lem}

\begin{proof}
	Since $f(A) = g(A)$ for all $A \in \Omega(U) $, this implies $\Omega(f) = \Omega(g)$. Hence, since 
	$\eta$ is a natural isomorphism,
	 \[
		f = \eta_{V}^{-1}   \pt(\Omega(f))   \eta_{U} = \eta_{V}^{-1}   \pt(\Omega(g))   \eta_{U} = g
	.\]
\end{proof}

\begin{lem} \label[lemma]{sobfund}
	Let $X$ be a sober space. Then, $\mathcal{I}(X,\tau)$ is fundamental.
\end{lem}

\begin{proof}
	Let $f,g \in \mathcal{I}(X, \tau) $ such that $(f,g) \in \mu$. Then, since $\mu $ is idempotent-separating, $f$ and $g$ 
	have same
	domain and range $f,g\colon U \to V$. Let $W \in \Omega(U)$, as $(f^{-1},g^{-1}) \in \mu$, 
	\[
		\id_{f(W)} = f \id_{W} f^{-1} = g \id_{W} g^{-1} = \id_{g(W)}
	.\]
	Hence, $f(W) = g(W)$ for all $W \in \Omega(U)$, by \Cref{lem: im gives equal} $f = g$.
\end{proof}

\begin{example}
	Recall from Example \ref{parti point} the space $X = \{0,1,2\}$ with the finite particular point topology.
	It was proved that $\mathcal{I}(X, \tau)$ was fundamental. However, $X$ is not sober. Note that
	$ \{ \{0\}, \{0,1\}, \{0,1,2\}\}$ is a completely prime filter of $\Omega(X)$, however it is not a neighbourhood filter
	of any point of $X$.
	Hence, $\eta_{X}\colon X \to \pt \Omega(X)$ is not surjective.
	This shows that if $\mathcal{I}(X, \tau)$ is fundamental it is not
	necessary that $X$ is sober.
\end{example}

\begin{prop} \label[proposition]{sober}
	Let $X$ be a sober space. Then, $T_{\Omega(X)} \cong \mathcal{I}(X, \tau)$.
\end{prop}

\begin{proof}
	We identify the idempotents $E(\mathcal{I}(X, \tau)) \cong \Omega(X)$.
By the Munn representation theorem and \Cref{sobfund} there is an embedding $\delta\colon \mathcal{I}(X, \tau)
	\to T_{\Omega(X)}$.
	We show that $\delta $ is surjective.  By \Cref{lem: image} for every
	order-isomorphism $\theta\colon \Omega(V) \to \Omega(U)$ there is a homeomorphism
	$\theta^{\ast }\colon U \to V$ such that $\theta^{\ast}(W) = \theta^{-1}(W)$ for all $W \in \Omega(V)$.
	Let $W \in \Omega(V)$.
	Then,
	\[
		\delta_{\theta^{\ast}}(\id_{W}) = (\theta^{\ast })^{-1} \id_{W} \theta^{\ast}
	\]
	is the identity on the set $(\theta^{\ast})^{-1}(W)$. Since $\theta $ is an isomorphism	$\theta^{\ast}(\theta(W)) = \theta^{-1}(\theta(W)) = W$. Hence,
	$\delta_{\theta^{\ast}}(\id_{W}) = \id_{\theta(W)}$ for all $W \in \Omega(V)$ and $\delta_{\theta^{\ast}} = \theta$.
\end{proof}

If $E$ is a frame with enough points then $\pt(E)$ is sober and $E \cong \Omega \pt (E)$.

\begin{cor}
	If $E$ is a frame with enough points then
	\[
		T_{E} \cong \mathcal{I}(\pt(E), \tau)
	.\]
\end{cor}

\begin{lem}
	Let $X$ and $Y$ be sober spaces such that $\mathcal{I}(X, \tau) \cong \mathcal{I}(Y, \tau)$. Then,
	$X$ is homeomorphic to $Y$.
\end{lem}
\begin{proof}
	Taking idempotents of $\mathcal{I}(X, \tau) \cong \mathcal{I}(Y, \tau)$, there is an order-isomorphism
	$\Omega(X) \cong \Omega(Y)$. Since order-isomorphisms are frame isomorphisms, by the adjunction in
	\Cref{adjunction} and the fact that $X$ and $Y$ are sober, there exists a homeomorphism $X \cong Y$.
\end{proof}

We can therefore view the Munn semigroup $T_{E}$ as a generalisation of $\mathcal{I}(X, \tau)$.
We now show that our definition of the generalised Munn semigroup plays a similar role, giving a frame theoretic version of the work
of Zhitormirskiy. 
In \cite{zhito1} the notion of partial homeomorphisms on a space is generalised to analogous maps on bundles.
In this paper we use the term
 \emph{(étale) bundle} to refer to a surjective local homeomorphism $\pi\colon X \to B$. Bundles are commonly known as étale spaces,
 see \cite[Sec. II.6]{sheaves}.
Fix a bundle $\pi\colon X \to  B$.
For each $U \subseteq B$ open, observe there is a 
bundle $\pi|_{\pi^{-1}(U)}\colon \pi^{-1}(U) \twoheadrightarrow U$ being the
restriction  of $\pi $ to $\pi^{-1}(U)$, we call this the \emph{principal subbundle} of $\pi$ at $U$. 
A map $\sigma\colon U \to X$ such that $\pi \sigma = \id_{U}$ for some open subset $U \subseteq B$ is called
a \emph{(local) section} of $\pi$.
A \emph{global section} of $\pi$ is a section with domain $B$.
By \cite[Chap. 2.3]{tennison1975sheaf} for any section $\sigma\colon U \to X$: the image of $\sigma$
is open in $X$, and  
$\sigma\colon U \to \sigma(U)$ is a homeomorphism with inverse $\pi|_{\sigma(U)}$,
since $\pi\colon X \to B$ is a local homeomorphism.
Denote the set of all sections of $\pi$ by
\[
	\Gamma(\pi) = \{\sigma\colon U \to X : U \in \Omega(X) \mbox{ and } \pi \sigma = \id_{U}\}
.\]
The set of all sections of $\pi$ forms a sheaf over $B$, see \cite{sheaves}, in particular this is a presheaf over
the semilattice $\Omega(B)$
satisfying additional compatibility conditions. 
Explicitly, the action of
the presheaf $\Gamma(\pi) \times \Omega(B) \to \Gamma(\pi)$ is given by $\sigma \cdot V = \sigma|_{V \cap U}$ 
with $U, V \in \Omega(B)$ and $\sigma\colon U \to X$. The map $p\colon \Gamma(\pi) \to  \Omega(B)$ is given by 
$p(\sigma) = \dom \sigma$.
Note that the presheaf has global support if and only if there exists a global section of $\pi$.
The principal 
subpresheaves of $(\Gamma(\pi), \Omega(B), p)$ are of the form
\[
	\Gamma(\pi) \cdot U
	= \{\sigma \in \Gamma(\pi): \dom \sigma \subseteq U\}.
\]
Note that the principal subpresheaf $\Gamma(\pi) \cdot U$ corresponds to the sheaf of sections of
the principal subbundle of $\pi$ at $U$.
For all $U \in \Omega(B)$ denote
the principal subpresheaf $\Gamma(\pi) \cdot U$ by $\Gamma U$. 

A \emph{principal partial automorphism of a bundle} of $\pi\colon X \to B$ is a 
pair of homeomorphisms 
$(\theta\colon \pi^{-1}(U) \to \pi^{-1}(V), \hat{\theta}\colon U \to V)$
such that the following diagram commutes:
\begin{equation} \label{comm bundle}
\begin{tikzcd}
	\pi^{-1}(U) \ar[r, "\theta"] \ar[d, "\pi"'] & \pi^{-1}(V) \ar[d, "\pi"] \\
	U \ar[r, "\hat{\theta }"] & V
\end{tikzcd}
\end{equation}
The set of all principal partial automorphisms of a bundle $\pi\colon X \to B$ defines an inverse semigroup $\mathbf{La}(\pi)$ where
composition is defined in both coordinates as the usual composition for partial homeomorphisms.
This was was defined by Zhitomirskiy in \cite{zhito1}. The goal of this section is to connect $\mathbf{La}(\pi)$ with
the generalised Munn semigroup of the presheaf $(\Gamma(\pi), \Omega(B), p)$, generalising \Cref{sober}.
We will require the space $B$ to be sober so we can apply the techniques we used before.
Begin by observing that the idempotents $E(\mathbf{La}(\pi))$ can be identified with the set of open sets of $B$ since
\[
	E(\mathbf{La}(\pi)) = \{(\id_{\pi^{-1}(U)}, \id_{U}) : U \in \Omega(B)\} \cong \Omega(B).
\]

\begin{prop} \label[proposition]{prop: act}
	Let $\pi\colon X \to B $ be a bundle. Then:
\begin{enumerate}
	\item 
	There exists
	a supported action $(\Gamma(\pi), \mathbf{La}(\pi), p)$ where the action
	$\Gamma(\pi)  \times \mathbf{La}(\pi) \to \Gamma(\pi)$ is
	defined by
	\[
	\sigma \cdot (\theta, \hat{\theta }) = \theta^{-1} \sigma \hat{\theta }\colon \hat{\theta }^{-1}(V \cap W) \to X,
	\]
	for $\hat{\theta }\colon U \to V$ and $\sigma\colon W \to X$. The map
	$p\colon \Gamma(\pi) \to  E(\mathbf{La}(\pi))$ is defined by
	$\sigma \mapsto (\id_{\pi^{-1}(W)}, \id_{W})$. Often we will just denote the action by 
	$\sigma \cdot \theta$.
	\item 
	The characteristic congruence of the supported action $(\Gamma(\pi), \mathbf{La}(\pi), p)$ is idempotent-separating.
	\item If $B$ is a sober space
	then $(\Gamma(\pi), \mathbf{La}(\pi), p)$ has characteristic congruence being equality.
\end{enumerate}
\end{prop}

\begin{proof}
	\quad

	\begin{enumerate}
	\item 
First, note that $\hat{\theta}^{-1} \sigma^{-1}\pi^{-1}(V) = \hat{\theta }^{-1}(V \cap W)$, so the function $ \theta^{-1} \sigma \hat{\theta }$ is well-defined on its
domain.
 We show that $\theta^{-1} \sigma \hat{\theta}^{-1}$ is 
a local section.
Let $x \in \dom \theta^{-1} \sigma \hat{\theta}$. Then, by the commutative diagram (\ref{comm bundle}),
 \[
	\pi \theta^{-1} \sigma \hat{\theta}(x) = \hat{\theta}^{-1} \pi \sigma \hat{\theta }(x) = \hat{\theta }^{-1} \hat{\theta}(x) = x
.\]
To show (SA1) let $(\phi, \hat{\phi }) \in \mathbf{La}(\pi)$ such that $\hat{\phi }\colon U' \to V'$. First it must be shown that the domains
of $( \sigma \cdot \phi) \cdot \theta $ and $ \sigma \cdot \phi \theta$ are equal:
\begin{equation*}
\dom (\sigma \cdot \phi \cdot \theta) = 
\hat{\theta}^{-1}( \hat{\phi }^{-1}(V' \cap W) \cap V) = (\widehat{ \phi \theta})^{-1}(W \cap \hat{\phi }(V \cap U')) 
	 = \dom (\sigma \cdot \phi \theta)
.\end{equation*}
Hence, 
\[
	(\sigma \cdot \phi) \cdot \theta = \theta^{-1}(\phi^{-1} \sigma \hat{\phi }) \hat{\theta } = (\phi \theta)^{-1} \sigma \widehat{ \phi \theta} = \sigma \cdot \phi \theta
.\]
Note that $\pi \sigma(U) = U$ so $\sigma(U) \subseteq \pi^{-1}(U)$. (SA2) follows from
\[
	\sigma \cdot p(U) = \id_{\pi^{-1}(U)} \sigma \id_{U} = \sigma
.\]
Since,
 $\hat{\theta}^{-1} \id_{W } \hat{\theta} $ is the identity of the set $\hat{\theta}^{-1}( \pi(W) \cap V)$
then
\[
p(	\sigma \cdot \theta) = (\id_{\pi^{-1}(\hat{\theta}^{-1}(W \cap V))}, \id_{\hat{\theta}^{-1}(W \cap V)})
= (\theta, \hat{\theta})^{-1} p(\sigma) (\theta, \hat{\theta})
\] and
(SA3) holds.
\item 
	We identify the idempotents of $\mathbf{La}(\pi)$ with the open sets $\Omega(B)$. Let $U,V \in \Omega(B)$ and
	suppose that they are equivalent under the characteristic congruence.  Let $x \in U$ then, since $\pi $ is
	a surjective local homeomorphism, there exists
	a section $\sigma_{x}\colon U_{x} \to X$ defined on an open neighbourhood $U_{x}$ of $x$. By definition of the 
	characteristic congruence
	\[
		 \sigma_{x}|_{U \cap U_{x}} = \sigma_{x} \cdot U  = \sigma_{x} \cdot V = \sigma_{x}|_{V \cap U_{x}}
	.\]
	Since $x \in U \cap U_{x} = V \cap U_{x}$ then $x \in V$. Repeating for all $x \in U$ it follows that $U \subseteq V$.
	A symmetric argument proves $V \subseteq U$ and therefore
	$U = V$.

\item 
	Let $(\theta, \hat{\theta})$ and $(\phi, \hat{\phi})$ be equivalent in $\mathbf{La}(\pi)$ under the characteristic congruence.
	Since the characteristic congruence is idempotent-separating
	they have the same domain and range. Suppose that $\theta, \phi\colon \pi^{-1}(U) \to \pi^{-1}(V)$.
	We now show that $\hat{\theta } = \hat{\phi }$ using the proof of \Cref{sobfund}. Again since the characteristic congruence is
	idempotent-separating then $((\theta, \hat{\theta}), (\phi, \hat{\phi})) \in \mu$. Hence, for all $W \in \Omega(U)$,
	\[
		(\theta, \hat{\theta})(\id_{\pi^{-1}(W)}, \id_{W})(\theta, \hat{\theta})^{-1} 
		 =  (\phi, \hat{\phi})(\id_{\pi^{-1}(W)}, \id_{W})(\phi, \hat{\phi})^{-1}
	\]
	and so
	\[
		(\id_{\theta(\pi^{-1}(W))}, \id_{\hat{\theta }(W)}) 
		= (\id_{\phi(\pi^{-1}(W))}, \id_{\hat{\phi }(W)})
	.\]
	Therefore, as $B$ is sober, $U$ and $V$ are sober so applying  \Cref{lem: im gives equal} proves $\hat{\theta } = \hat{\phi}$.
	Now we show that $\theta = \phi$.
	Fix an element $x \in \pi^{-1}(U)$. Since $\pi$ is a local homeomorphism there exists a section $\sigma\colon W \to X$ such that
$x \in \sigma(W) \subseteq \pi^{-1}(U)$.
	By assumption
	\[
		\theta \sigma \hat{\theta}^{-1} = \sigma \cdot (\theta, \hat{\theta})^{-1} = \sigma \cdot (\phi, \hat{\phi})^{-1} = \phi \sigma \hat{\phi}^{-1}.
	\]
	Since $\sigma(W) \subseteq \pi^{-1}(U)$ then $W \subseteq U$, using this and that $\hat{\theta}=\hat{\phi}$ the domain of the section $\theta \sigma \hat{\theta}^{-1}$ is
	 $\hat{\theta}(U \cap W) = \hat{\theta}(W) = \hat{\phi }(W)$.
	 Hence,  $\theta(\sigma(W)) = \phi(\sigma(W))$.
	By assumption $\sigma $ is a local homeomorphism on to its image, since $B$ is sober then $W $ is sober and homeomorphic to $\sigma(W)$. Therefore,
	$\theta|_{\sigma(W)}$ and $\phi|_{\sigma(W)}$ are both homeomorphisms
	between sober spaces.
	For any open set $A$ in $\Omega(\sigma(W)) $, $\sigma|_{\pi(A)}$ is a local section of $\pi$ and $ \theta(A) = \theta(\sigma(\pi(A))) = \phi(\sigma(\pi(B))) = \phi(A)$ as above.
	By \Cref{lem: im gives equal}
	$\theta|_{W} = \phi|_{W}$ and repeating this for all $x \in \pi^{-1}(U)$ we deduce that $\theta = \phi$.

	\end{enumerate}
\end{proof}

\begin{thm}
	Let $\pi\colon X \to B$ be an étale bundle such that $B$ is a sober space. Then,
	\[
		T_{\Gamma(\pi)} \cong \mathbf{La}(\pi),
	\]
	where $T_{\Gamma(\pi)}$ is the generalised Munn semigroup associated to the sheaf of sections $\Gamma(\pi)$ of $\pi$,
	and $\mathbf{La}(\pi)$ is the inverse semigroup 
	of principal partial automorphisms of $\pi$.
\end{thm}

\begin{proof}
	By \Cref{prop: act} the characteristic congruence of $(\Gamma(\pi), \mathbf{La}(\pi), p)$ is equality hence, by
	\Cref{IdempRho}, the generalised Munn representation of \Cref{Thm: General Munn} is an embedding
	 $\xi\colon\mathbf{La}(\pi) \to T_{\Gamma(\pi)}$. We show that this
	is also surjective.
Using the identification $E(\mathbf{La}(\pi)) \cong \Omega(B)$ let
\[
	(\alpha\colon \Gamma V \to \Gamma U, \theta_{\alpha }\colon \Omega(V) \to \Omega(U))
	\in T_{\Gamma(\pi)}.
\]
Since $B $ is sober then by \Cref{lem: image} there is a homeomorphism $\theta_{\alpha}^{\ast}\colon U \to V$ such that
$\theta_{\alpha}^{\ast}(W) = \theta_{\alpha}^{-1}(W)$ for all $W \in \Omega(U)$. We wish to construct a homeomorphism
$\alpha^{\ast}\colon \pi^{-1}(U) \to \pi^{-1}(V)$ such that  $(\alpha^{\ast}, \theta_{\alpha}^{\ast}) \in \mathbf{La}(\pi)$.
Let $x \in \pi^{-1}(U)$. Since $\pi$ is a local homeomorphism there exists a local section $\sigma\colon W \to X$ such that
$x \in \sigma(W) \subseteq \pi^{-1}(U)$. The map $\alpha^{-1}(\sigma)\colon p(\alpha^{-1}(\sigma)) \to X$ is a section of $\pi$.
Observe that
\[
	\theta_{\alpha}^{\ast}(\pi(x)) \in \theta_{\alpha}^{\ast}(W) =  \theta_{\alpha}^{-1}(p(\sigma)) = p(\alpha^{-1}(\sigma)).
\]
Define $\alpha^{\ast}(x)$  to be the image of $\theta_{\alpha}^{\ast}(\pi(x))$ under the section $\alpha^{-1}(\sigma)$, we show this does not depend on the section $\sigma $ chosen.
Let $\tau\colon W' \to X$ be another section such that 
$x \in \tau(W') \subseteq \pi^{-1}(U)$.
Then, $\tau $ and $\sigma $ agree at the point $\pi(x)$, hence they agree on the set $Z \coloneqq \pi(\tau(W') \cap \sigma(W))$.
Since $\alpha $ is an order-isomorphism then
\[
	\alpha^{-1}(\tau) \cdot \theta_{\alpha }^{-1}(Z) =   \alpha^{-1}(\tau \cdot Z) =  \alpha^{-1}( \sigma \cdot Z) =  \alpha^{-1}( \sigma) \cdot \theta^{-1}_{\alpha }(Z)
\]
and $\pi(x) \in Z$ implies that the image of $\theta_{\alpha }^{\ast}(\pi(x)) $ under $\alpha^{-1}(\tau)$ and $\alpha^{-1}(\sigma)$ agree.
We show that the function $\alpha^{\ast}$ is injective. Let $x,y \in \pi^{-1}(U)$ such that
$\alpha^{\ast}(x) = \alpha^{\ast}(y)$. Then, there exist sections  $\sigma\colon W_{x} \to X$ and $\tau\colon W_{y } \to X$ such that
$\sigma(\pi(x)) = x$ and $\tau(\pi(y)) = y$. Since the sections
$\alpha^{-1}(\tau)$ and $\alpha^{-1}(\sigma)$ agree at $\theta_{\alpha}^{\ast}(\pi(x)) = \theta_{\alpha}^{\ast}(\pi(y))$ then, as before, they agree on an open set
$Z$ containing  $\theta_{\alpha}^{\ast}(\pi(x))$. Therefore, $\sigma $ and $\tau $ agree on an open set containing 
$\pi(x) = \pi(y)$ and $x = \sigma(\pi(x)) = \tau(\pi(y)) = y$.
We show the function $\alpha^{\ast}$ is surjective. Let $x \in \pi^{-1}(V)$. Then, there exists a section $\sigma\colon
W \to X$ of $\pi$ such that
$\sigma(\pi(x)) = x$. Let $y$ be the image of $(\theta_{\alpha}^{\ast})^{-1}(\pi(x)) $ under the section $\alpha(\sigma)$. It follows that
$\alpha^{\ast}(y) = x$.
Therefore, $\alpha^{\ast}$ is a bijective function and by definition $\pi \alpha^{\ast} = \pi \theta_{\alpha}$. It is only left
to show that 
$\alpha^{\ast}$ is open and continuous. Here we use the fact that $X$ has a basis of open sets
$ \{\ima \sigma : \sigma \in \Gamma(\pi)\}$ and it is enough to check continuity on this basis.
Let $\sigma\colon W \to X$ be a local section such that $\sigma(W) \subseteq p^{-1}(U)$. Then, by definition,
\[
	\alpha^{\ast}(\ima \sigma) \subseteq \ima \alpha^{-1}(\sigma)
.\]
For the other containment let $x \in \ima \alpha^{-1}(\sigma)$, then $\sigma((\theta_{\alpha}^{\ast})^{-1}(\pi(x))) \in \ima \sigma$
and hence  $x \in \alpha^{\ast}(\ima \sigma)$.
 Since $\alpha$ and $\alpha^{\ast}$ are bijections we also find that
$(\alpha^{\ast})^{-1}(\ima \sigma) = \ima \alpha(\sigma)$ for all $\sigma \in \Gamma(\pi)$ such that 
$\ima \sigma \subseteq \pi^{-1}(V)$.
Hence,  $\alpha^{\ast}$ is a homeomorphism.

We have constructed $(\alpha^{\ast}, \theta_{\alpha}^{\ast}) \in \mathbf{La}(\pi)$. The homomorphism 
$\xi $ maps the element $(\alpha^{\ast}, \theta_{\alpha }^{\ast})^{-1}$ to $(\beta, \theta_{\alpha})$ where
the map  $\beta\colon \Gamma V \to \Gamma U$ is defined by
$\sigma \mapsto (\alpha^{\ast})^{-1} \sigma \theta^{\ast}_{\alpha}$.
The domain of $(\alpha^{\ast})^{-1} \sigma \theta^{\ast}_{\alpha}$ is
\[
	(\theta_{\alpha}^{\ast})^{-1}( \dom \sigma) = \theta_{\alpha}(p(\sigma)) = p \alpha(\sigma) = \dom \alpha( \sigma).
\]
Let $x \in \dom \alpha(\sigma)$. Then,
\[
	\alpha^{\ast}(\alpha(\sigma)(x)) = \alpha^{-1}\alpha(\sigma)(\theta_{\alpha}^{\ast}(x)) = \sigma(\theta_{\alpha}^{\ast}(x))
.\]
 Therefore,
$((\alpha^{\ast})^{-1}\sigma\theta^{\ast}_{\alpha})(x) = \alpha(\sigma)(x)$ and
$\xi((\alpha^{\ast}, \theta_{\alpha }^{\ast})^{-1}) = (\beta, \theta_{\alpha}) = (\alpha, \theta_{\alpha })$.
\end{proof}

\section{Bundles from presheaves}

In the previous section we showed how every bundle gives rise to a sheaf. We now look at how presheaves give rise 
to bundles.
Let $S$ be an inverse semigroup
and $(X,S,p)$ a globally supported action.
It was first noted by Steinberg \cite{steinberg2009strong} that for such an action we can define respective topologies
on $X$ and $E(S)$ by
letting their open sets be order-ideals and then
the map $p\colon X \to E(S)$ is a local homeomorphism. Moreover, following Wagner \cite[Chapter 5.2]{LawsonBook} both
$E(S)$ and $X$ are $T_0$ under this topology. Since the action has global support then $p\colon X \to E(S)$ is a bundle.
Now as seen in \Cref{prop: act} from every bundle one may construct a supported action with global support. We would
like to know what are the sections of $p\colon X \to E(S)$. Following \cite{lawson2021morita}
define two elements
$x,y \in X$ to be \emph{compatible} if $x \cdot p(y) = y \cdot p(x)$. A subset of $X$ is called \emph{compatible}
if it pairwise compatible.

\begin{lem}
	Let $S$ be an inverse semigroup and 
$(X,S,p)$ a globally supported action. Then, $p|_{U}\colon U \to p(U)$ is a homeomorphism, 
 with respect to the topology above, if and only
if $U$ is a compatible order-ideal.
\end{lem}

\begin{proof}
	 Let $U$ be a compatible order-ideal of $X$ and
	$x,y \in U$ such that $p(x) = p(y)$. Then, since $x$ and $y$ are compatible,
	 \[
		x = x \cdot p(x) = x \cdot p(y) = y \cdot p(x) = y \cdot p(y) = y
	.\] Hence, since $p$ is a local homeomorphism $p|_{U}\colon U \to p(U)$ is a homeomorphism. 
	Now suppose $V$ is an order-ideal of $X$, equivalently an open set, such that $p|_{V}$ is injective. 
	Let $x,y \in V$. Since $V$ is an order-ideal
	$x \cdot p(y)$ and $y \cdot p(x)$ are in $V$. By assumption 
	\[
		p(x \cdot p(y)) = p(x)p(y) = p(y \cdot p(x))
	\]
	implies that $x \cdot p(y) = y \cdot p(x)$. So $V$ is a 
	compatible order-ideal of $X$.
\end{proof}

Since $p$ is a surjective local homeomorphism the sections of $p$ are identified with their images, which is exactly
the open sets where $p$ is injective. Hence, the set of sections of $p$ are in bijection with the set of compatible order-ideals.
We remark that taking all the compatible order-ideals of a supported action
is the pseudogroup module completion of
\cite{lawson2021morita}.


\section*{Acknowledgments} This work was supported by the UKRI Centre for Doctoral Training in Algebra, Geometry and Quantum Fields (AGQ), Grant Number EP/Y035232/1.
I would like to thank my supervisor, Mark Lawson, for his significant contributions to this work and his many helpful suggestions during its preparation.

\end{document}